\documentclass[a4paper,12pt,reqno]{amsart}

\usepackage[T1]{fontenc}
\usepackage[utf8]{inputenc}
\usepackage{lmodern}
\usepackage[english]{babel}
\usepackage{xcolor}

\usepackage[%
	left=2.5cm,       
	right=2.5cm,      
	top=3.5cm,        
	bottom=3.5cm,     
	heightrounded,    
	bindingoffset=0mm 
]{geometry}

\usepackage{amsmath}
\usepackage{amssymb}
\usepackage{mathtools}
\usepackage{mathrsfs}
\usepackage[normalem]{ulem}

\numberwithin{equation}{section}

\usepackage{hyperref} 
\usepackage[noabbrev,capitalize]{cleveref}

\usepackage{bookmark}

\usepackage[initials
]{amsrefs}

\usepackage{enumerate}
\usepackage{url}

\theoremstyle{plain}
	\newtheorem{theorem}{Theorem}[section]
	\newtheorem{lemma}[theorem]{Lemma}
	\newtheorem{proposition}[theorem]{Proposition}
	\newtheorem{corollary}[theorem]{Corollary}

\theoremstyle{definition}
	\newtheorem{definition}[theorem]{Definition}
	
	\newtheorem{remark}[theorem]{Remark}
	\newtheorem{open.problem}[theorem]{Open Problem}

\newcommand{\N}{\mathbb{N}}
\newcommand{\R}{\mathbb{R}}
\newcommand{\eps}{\varepsilon}
\newcommand{\M}{\mathscr{M}}
\newcommand{\Haus}[1]{\mathscr{H}^{#1}} 
\newcommand{\Leb}[1]{\mathscr{L}^{#1}} 
\newcommand{\redb}{\mathscr{F}} 
\newcommand{\closure}[2][3]{{}\mkern#1mu\overline{\mkern-#1mu#2}}

\newcommand{\weakto}{\rightharpoonup}

\renewcommand{\phi}{\varphi}
\renewcommand{\rho}{\varrho}
\renewcommand{\theta}{\vartheta}
\renewcommand{\div}{\mathrm{div}} 

\DeclareMathOperator{\loc}{loc}
\DeclareMathOperator{\NL}{NL}
\DeclareMathOperator{\Capa}{Cap}
\DeclareMathOperator{\BMO}{BMO}
\DeclareMathOperator{\Lip}{Lip}
					    
\DeclarePairedDelimiter{\scalar}{<}{>}                                     
\DeclarePairedDelimiter{\set}{\{}{\}}

\mathchardef\ordinarycolon\mathcode`\:
\mathcode`\:=\string"8000
\begingroup \catcode`\:=\active
  \gdef:{\mathrel{\mathop\ordinarycolon}}
\endgroup

\def\Xint#1{\mathchoice
{\XXint\displaystyle\textstyle{#1}}%
{\XXint\textstyle\scriptstyle{#1}}%
{\XXint\scriptstyle\scriptscriptstyle{#1}}%
{\XXint\scriptscriptstyle\scriptscriptstyle{#1}}%
\!\int}
\def\XXint#1#2#3{{\setbox0=\hbox{$#1{#2#3}{\int}$ }
\vcenter{\hbox{$#2#3$ }}\kern-.6\wd0}}

\def\aint{\Xint-}

\allowdisplaybreaks

\begin{document}

\title[Leibniz rules and Gauss--Green formulas in fractional spaces]{Leibniz rules and Gauss--Green formulas in distributional fractional spaces}

\author[G.~E.~Comi]{Giovanni E. Comi}
\address[G.~E.~Comi]{Dipartimento di Matematica, Universit\`a di Pisa, Largo Bruno Pontecorvo 5, 56127 Pisa, Italy}
\email{giovanni.comi@dm.unipi.it}

\author[G.~Stefani]{Giorgio Stefani}
\address[G.~Stefani]{Department Mathematik und Informatik, Universit\"at Basel, Spiegelgasse 1, CH-4051 Basel, Switzerland}
\email{giorgio.stefani.math@gmail.com}

\thanks{\textit{Acknowledgments}.
The authors are members of INdAM--GNAMPA.
The first author is partially supported by the PRIN 2017 Project \textit{Variational methods for stationary and evolution problems with singularities and interfaces} (n.\ prot.\ 2017BTM7SN).   
The second author is partially supported by the ERC Starting Grant 676675 FLIRT -- \textit{Fluid Flows and Irregular Transport} and by the INdAM--GNAMPA Project 2020 \textit{Problemi isoperimetrici con anisotropie} (n.\ prot.\ U-UFMBAZ-2020-000798 15-04-2020).
The second author acknowledges the Department of Mathematics of the University of Pisa (Italy), where part of this work was completed, for its support and the warm hospitality.
The authors wish to thank the anonymous referee for their comments.}



\date{\today}

\begin{abstract}
We apply the results established in~\cite{CSS21} to prove some new fractional Leibniz rules involving $BV^{\alpha,p}$ and $S^{\alpha,p}$ functions, following the distributional approach adopted in the previous works~\cites{BCCS20,CS19,CS19-2}. 
In order to achieve our main results, we revise the elementary properties of the fractional operators involved in the framework of Besov spaces and we rephraze the Kenig--Ponce--Vega Leibniz-type rule in our fractional context. We apply our results to prove the well-posedness of the boundary-value problem for a general $2\alpha$-order fractional elliptic operator in divergence form.
\end{abstract}

\maketitle

\tableofcontents

\section{Introduction}

\subsection{The fractional framework}

Given a parameter $\alpha\in(0,1)$, the \emph{fractional $\alpha$-gradient} and the \emph{fractional $\alpha$-divergence}
are respectively the operators defined as
\begin{equation}
\label{eqi:def_frac_nabla}
\nabla^\alpha f(x)
=
\mu_{n,\alpha}
\int_{\R^n}\frac{(y-x)(f(y)-f(x))}{|y-x|^{n+\alpha+1}}\,dy,
\quad
x\in\R^n,	
\end{equation}
and
\begin{equation}
\label{eqi:def_frac_div}
\div^\alpha\phi(x)
=
\mu_{n,\alpha}
\int_{\R^n}\frac{(y-x)\cdot(\phi(y)-\phi(x))}{|y-x|^{n+\alpha+1}}\,dy,
\quad
x\in\R^n,	
\end{equation}
where $\mu_{n,\alpha}$ is a suitable renormalizing constant depending on $n$ and $\alpha$ only.

Since its first appearance~\cite{H59}, the literature around the operator $\nabla^\alpha$ has been quickly expanding in various directions, such as the study of PDEs~\cites{lo2021class, rodrigues2019nonlocal,rodrigues2019nonlocal-corr,SS15,SS18,SSS15,SSS18} and of functionals~\cites{BCM20,BCM21,KS21} involving fractional operators, the discovery of new optimal embedding inequalities~\cites{SSS17,S19,S20-An} and the development of a distributional and asymptotic analysis in this fractional framework~\cites{BCCS20,CS19,CS19-2,Sil19,CSS21}.
We also refer the reader to the survey~\cite{S20-New} and to the monograph~\cite{P16}.

At least for sufficiently smooth functions, the operators $\nabla^\alpha$ and $\div^\alpha$ are \emph{dual}, in the sense that
\begin{equation}
\label{eqi:ibp}
\int_{\R^n}f\,\div^\alpha\phi \,dx=-\int_{\R^n}\phi\cdot\nabla^\alpha f\,dx.
\end{equation}
Equality~\eqref{eqi:ibp} can be regarded as a fractional integration-by-parts formula and provides  the starting point for the distributional theory developed in the previous papers~\cites{BCCS20,CS19,CS19-2,CSS21}.

By imitating the classical definition of $BV$ function, for a given exponent $p\in[1,+\infty]$, we say that a function $f\in L^p(\R^n)$ belongs to the space $BV^{\alpha,p}(\R^n)$ if
\begin{equation}
\label{eqi:def_frac_var_p}
|D^\alpha f|(\R^n)
=
\sup\set*{\int_{\R^n}f\,\div^\alpha\phi\,dx : \phi\in C^\infty_c(\R^n;\R^n),\ \|\phi\|_{L^\infty(\R^n;\,\R^n)}\le1}
<+\infty.
\end{equation}	
In the case $p=1$, we simply write $BV^{\alpha,1}(\R^n)=BV^\alpha(\R^n)$.
The resulting linear space 
\begin{equation*}
BV^{\alpha,p}(\R^n)=
\set*{f\in L^p(\R^n) : |D^\alpha f|(\R^n)<+\infty}
\end{equation*}
endowed with the norm
\begin{equation*}
\|f\|_{BV^{\alpha,p}(\R^n)}
=
\|f\|_{L^p(\R^n)}+|D^\alpha f|(\R^n),
\quad
f\in BV^{\alpha,p}(\R^n),
\end{equation*}
is a Banach space and the fractional variation defined in~\eqref{eqi:def_frac_var_p} is lower semicontinuous with respect to the $L^p$-convergence. The fractional variation was considered by the authors in the work~\cite{CS19} in the \emph{geometric} framework $p=1$, also in connection with the naturally associated notion of \emph{fractional Caccioppoli perimeter}.
The fractional variation of an $L^p$ function for an arbitrary exponent $p\in[1,+\infty]$ was then explored in~\cites{CS19-2,CSS21}.

By instead emulating the usual definition of Sobolev function, we say that a function $f\in L^p(\R^n)$ belongs to the space $S^{\alpha,p}(\R^n)$ if there exists a function $\nabla^\alpha f\in L^p(\R^n;\R^n)$ satisfying~\eqref{eqi:ibp} for all $\phi\in C^\infty_c(\R^n;\R^n)$.
The resulting linear space
\begin{equation*}
S^{\alpha,p}(\R^n)
=
\set*{f\in L^p(\R^n): \nabla^\alpha f\in L^p(\R^n;\R^n)}
\end{equation*}
endowed with the norm
\begin{equation*}
\|f\|_{S^{\alpha,p}(\R^n)}
=
\|f\|_{L^p(\R^n)}
+
\|\nabla^\alpha f\|_{L^p(\R^n;\,\R^n)},
\quad
f\in S^{\alpha,p}(\R^n),
\end{equation*}  
is a Banach space as well and, actually, can be identified with the \emph{Bessel potential space} $L^{\alpha,p}(\R^n)$ whenever $p\in(1,+\infty)$, see~\cites{BCCS20,KS21}. 

\subsection{The fractional Leibniz rule: local and non-local parts}

Playing with its definition~\eqref{eqi:def_frac_nabla} (see~\cite{CS19}*{Lemmas~2.6 and~2.7} and~\cite{CS19-2}*{Lemmas~2.4 and~2.5} for the rigorous statements), one immediately sees that  the fractional gradient $\nabla^\alpha$ obeys the Leibniz-type rule
\begin{equation}
\label{eqi:leibniz_frac_nabla}
\nabla^\alpha(fg)
=
g\,\nabla^\alpha f
+
f\,\nabla^\alpha g
+
\nabla^\alpha_{\NL}(f,g)
\end{equation}  
at least for $f,g\in C^\infty_c(\R^n)$.
The remainder term $\nabla^\alpha_{\NL}(f,g)$ is the \emph{non-local fractional gradient} of the couple $(f,g)$ and is defined as
\begin{equation}
\label{eqi:def_frac_nabla_NL}
\nabla^\alpha_{\NL} (f,g)(x)
=
\mu_{n,\alpha}
\int_{\R^n}\frac{(y-x)(f(y)-f(x))(g(y)-g(x))}{|y-x|^{n+\alpha+1}}\,dy,
\quad
x\in\R^n.	
\end{equation}
In contrast with the more familiar local part given by $g\,\nabla^\alpha f+f\,\nabla^\alpha g$, the operator in~\eqref{eqi:def_frac_nabla_NL} is reminiscent of the intrinsic non-local nature of the fractional gradient~$\nabla^\alpha$.
An analogous Leibniz-type rule can be obtained for the fractional divergence operator~\eqref{eqi:def_frac_div}, whose reminder term $\div^\alpha_{\NL}$ is the \emph{non-local fractional divergence} and is defined similarly to~\eqref{eqi:def_frac_nabla_NL}.

When dealing with less regular functions, there are two main obstacles for the validity of~\eqref{eqi:leibniz_frac_nabla} naturally arising from its local and non-local parts respectively.

On the one side, in analogy with what already happens in the classical framework, one expects that the local term $f\,\nabla^\alpha g+g\,\nabla^\alpha f$ should be replaced by the sum of two distributionally defined objects.
In particular, if $f$ and $g$ have finite fractional variation~\eqref{eqi:def_frac_var_p}, then  it is natural to replace the fractional gradients $\nabla^\alpha f,\nabla^\alpha g\in L^1(\R^n;\R^n)$ with the fractional variation measures $D^\alpha f,D^\alpha g\in\M(\R^n;\R^n)$, but one has to check that the resulting measures $g\,D^\alpha f$ and $f\,D^\alpha g$ are well defined.

On the other side, one has to deal with the well-posedness and the integrability properties of (the extension of) the bilinear operator~\eqref{eqi:def_frac_nabla_NL} appearing as a remainder in the Leibniz rule~\eqref{eqi:leibniz_frac_nabla}.
In particular, the elementary estimate
\begin{equation}
\label{eqi:trivial_bound_NL}
|\nabla^\alpha_{\NL}(f,g)(x)|
\le
\mu_{n,\alpha}
\int_{\R^n}
\frac{|f(y)-f(x)|\,|g(y)-g(x)|}{|y-x|^{n+\alpha}}\,dy,
\quad
x\in\R^n,
\end{equation}
suggests that one needs to take the interplay between the two fractional regularities of the functions~$f$ and~$g$ into consideration.

Having in mind the development of the PDE theory~\cites{lo2021class,rodrigues2019nonlocal,rodrigues2019nonlocal-corr, SS15,SS18,SSS15,SSS18} and the study of functionals~\cites{BCM20,BCM21,KS21} in this fractional setting, as well as the large collection of Leibniz-type rules for fractional operators~\cites{BMMN14,BL14,CW91,d2019short,d2020correction,G12,GMN14,GK96,GO14,KP88,MS13,OW20,T06}, the main aim of the present paper is to move a first step towards the solution of these two problems, providing new Leibniz rules for the pairing of functions having finite fractional variation as well as for Bessel functions. As an application of our extended Leibniz rules, we show the well-posedness of a general class of $2\alpha$-order fractional elliptic boundary-value problems.

\subsection{Dealing with the local part}

To treat the issues coming from the local part of the fractional Leibniz rule~\eqref{eqi:leibniz_frac_nabla}, we take inspiration from the classical integer case~$\alpha=1$.

As a matter of fact, it is well known that, if $f,g\in BV^{1, \infty}(\R^n)$, then $fg\in BV^{1, \infty}(\R^n)$ with 
\begin{equation}
\label{eqi:leibniz_BV}
D(fg)
=
g^\star\,Df
+
f^\star\,Dg
\quad
\text{in}\
\M(\R^n;\R^n),
\end{equation}
where $BV^{1, \infty}(\R^n) = \set*{f\in L^{\infty}(\R^n) : |D f|(\R^n)<+\infty}$.
Here and in the following, for a function $f\in L^1_{\loc}(\R^n)$, we let 
\begin{equation}
\label{eqi:def_prec_rep}
f^\star(x)
=
\lim_{r\to0^+}
\aint_{B_r(x)}
f(y)\,dy
\end{equation} 
for all $x\in\R^n$ such that the limit~\eqref{eqi:def_prec_rep} exists and $f^\star(x)=0$ otherwise. 
The function $f^\star$ defined by~\eqref{eqi:def_prec_rep} is the so-called \emph{precise representative} of the function~$f$.
If $f\in BV_{\rm loc}(\R^n)$, then the limit in \eqref{eqi:def_prec_rep} exists for $\Haus{n-1}$-a.e.\ $x \in \R^n$ and it can be strengthened as
\begin{equation*}
\lim_{r\to0^+}
\aint_{B_r(x)}
|f(y)-f^\star(x)|^{\frac n{n-1}}\,dy=0
\end{equation*}
for $\Haus{n-1}$-a.e.\ $x\in\R^n\setminus J_f$, see~\cite{EG15}*{Section~5.9} for example, where $J_f\subset\R^n$ is the so-called \emph{jump set} of the function $f\in BV(\R^n)$ (if $f\in W^{1,1}(\R^n)$, then $J_f$ is empty).
Since $|Df|\ll\Haus{n-1}$ for all $f\in BV_{\rm loc}(\R^n)$, the Leibniz rule in~\eqref{eqi:leibniz_BV} is well posed. 

We would like to emphasize that the role of the precise representative is obviously relevant uniquely in the presence of a measure.
In fact, for Sobolev functions, it is much easier to see that, if $f,g\in W^{1,p}(\R^n)\cap L^\infty(\R^n)$ for some $p\in[1,+\infty]$, then $fg\in W^{1,p}(\R^n)\cap L^{\infty}(\R^n)$ with
\begin{equation*}
\nabla(fg)
=
f\,\nabla g
+
g\,\nabla f
\quad
\text{in}\
L^p(\R^n;\R^n).
\end{equation*}
A similar result clearly holds also for $f\in W^{1,p}(\R^n)$ and $g\in W^{1,q}(\R^n)$ with $p,q\in[1,+\infty]$ such that $\frac1p+\frac1q\le1$, in which case $fg \in W^{1,r}(\R^n)$, with $r \in [1, +\infty]$ such that $\frac{1}{r} = \frac{1}{p} + \frac{1}{q}$.

Reasoning by analogy, we will have to face the issue linked to the well-posedness of the local part of the fractional Leibniz rule~\eqref{eqi:leibniz_frac_nabla} only when dealing with functions with $BV^{\alpha,p}$ regularity. 
The difficulties related to  $S^{\alpha,p}$ functions, instead, will depend uniquely on the non-local part of the formula.

In order to control the continuity properties of the fractional variation with respect to the Hausdorff measure, we exploit~\cite{CSS21}*{Theorem~1.1}, which states that, if $f\in BV^{\alpha,p}(\R^n)$, then
\begin{equation}
\label{eqi:abs_frac_var}
|D^\alpha f|
\ll
\begin{cases}
\Haus{n-1}
&
\text{if}\
p\in\left[1,\frac n{1-\alpha}\right),
\\[3mm]
\Haus{n-\alpha-\frac np}
&
\text{if}\
p\in\left[\frac n{1-\alpha},+\infty\right].
\end{cases}
\end{equation}
Consequently, the product $g^\star\,D^\alpha f$ is well posed as soon as $g^\star$ is well defined as the limit in \eqref{eqi:def_prec_rep} in term of a Hausdorff measure of dimension not larger than the ones appearing in~\eqref{eqi:abs_frac_var}.
In particular, the function~$g$ is allowed to belong to some Besov space with sufficiently large regularity order.
However, at least at the present stage of the development of the theory, we are not able to consider functions~$g$ with finite fractional variation since, in virtue of~\cite{CSS21}*{Theorem~1.2}, $g^\star(x)$ would be defined only for $\Haus{n-\alpha+\eps}$-a.e.\ $x\in\R^n$, for any $\eps>0$ arbitrarily small.

\subsection{Dealing with the non-local part}

To overcome the issue related to the non-local part of the fractional Leibniz rule~\eqref{eqi:leibniz_frac_nabla}, we need to understand the behavior of the non-local fractional gradient~\eqref{eqi:def_frac_nabla_NL}.

A first step in this direction was already made in~\cite{CS19}*{Lemma~2.6} and~\cite{CS19-2}*{Lemma~2.4} by the authors, who proved that 
\begin{equation}
\label{eqi:cs_frac_nabla_NL_W}
\|\nabla^\alpha_{\NL}(f,g)\|_{L^1(\R^n;\,\R^n)}
\le
\begin{cases}
\mu_{n,\alpha}
\,[f]_{W^{\frac\alpha p,p}(\R^n)}
\,
[g]_{W^{\frac\alpha q,q}(\R^n)}
\\[2mm]
2\mu_{n,\alpha}
\,
\|f\|_{L^\infty(\R^n)}
\,
[g]_{W^{\alpha,1}(\R^n)}
\\[2mm]
2\mu_{n,\alpha}
\,
[f]_{W^{\alpha,1}(\R^n)}
\,
\|g\|_{L^\infty(\R^n)}
\end{cases}
\end{equation}
whenever $p,q\in(1,+\infty)$ are such that $\frac1p+\frac1q=1$.
In addition, from the computations made in~\cite{KS21}*{Equation~(2.11)}, one can also deduce the bound
\begin{equation}
\label{eqi:ks_frac_nabla_NL_Lp}
\|\nabla^\alpha_{\NL}(f,g)\|_{L^p(\R^n;\,\R^n)}
\le
\begin{cases}
2 \mu_{n,\alpha}
\,[f]_{B^\alpha_{\infty,1}(\R^n)}
\,
\|g\|_{L^p(\R^n)}
\\[2mm]
2 \mu_{n,\alpha}
\,
\|f\|_{L^p(\R^n)}
\,
[g]_{B^\alpha_{\infty,1}(\R^n)}
\end{cases}
\end{equation}
for all $p\in[1,+\infty]$.

In view of the simple bound~\eqref{eqi:trivial_bound_NL},  inequalities~\eqref{eqi:cs_frac_nabla_NL_W} and~\eqref{eqi:ks_frac_nabla_NL_Lp} suggest that one should study the integrability properties of the operator
\begin{equation*}
\mathcal D^\alpha_{\NL}(f,g)(x)
=
\int_{\R^n}
\frac{|f(y)-f(x)|\,|g(y)-g(x)|}{|y-x|^{n+\alpha}}\,dy,
\quad
x\in\R^n,
\end{equation*}  
in the general framework of Besov spaces.
A simple application of Minkowski's integral inequality (see for instance \cite{G14-C}*{Exercise 1.1.6}) and then H\"older's inequality easily shows that
\begin{equation}
\label{eqi:frac_nabla_NL_Besov}	
\|
\mathcal D^\alpha_{\NL}(f,g)
\|_{L^r(\R^n)}
\le
\begin{cases}
[f]_{B^{\beta}_{p,s}(\R^n)}
\,
[g]_{B^{\gamma}_{q,t}(\R^n)}
\\[2mm]
2
\,
\|f\|_{L^p(\R^n)}
\,
[g]_{B^\alpha_{q,1}(\R^n)}
\\[2mm]
2
\,
[f]_{B^\alpha_{p,1}(\R^n)}	
\,
\|g\|_{L^q(\R^n)}
\end{cases}
\end{equation}   
whenever $\beta,\gamma\in(0,1)$ are such that $\alpha=\beta+\gamma$ and $p,q,r,s,t\in[1,+\infty]$ are such that $\frac1p+\frac1q=\frac1r$ and $\frac1s+\frac1t=1$, including~\eqref{eqi:cs_frac_nabla_NL_W} and~\eqref{eqi:ks_frac_nabla_NL_Lp} as particular cases. 
Note that one can argue in a similar way for the fractional gradient $\nabla^\alpha$ in~\eqref{eqi:def_frac_nabla}, since the related operator
\begin{equation*}
\mathcal D^\alpha f(x)
=
\int_{\R^n}\frac{|f(y)-f(x)|}{|y-x|^{n+\alpha}}\,dy,
\quad
x\in\R^n,
\end{equation*}
satisfies
\begin{equation}
\label{eqi:frac_nabla_Besov}
\|\mathcal D^\alpha f\|_{L^p(\R^n)}
\le
[f]_{B^\alpha_{p,1}(\R^n)}
\end{equation}
for all $p\in[1,+\infty]$, again by Minkowski's integral inequality.

On the one side, the elementary bounds~\eqref{eqi:frac_nabla_NL_Besov} and~\eqref{eqi:frac_nabla_Besov} can be taken as a new starting point for the well-posedness of the fractional theory related to the operators~$\nabla^\alpha$ and~$\nabla^\alpha_{\NL}$ (and their companions $\div^\alpha$ and~$\div^\alpha_{\NL}$) in the setting of Besov spaces.
This point of view has the advantage of simplifying the exposition of the theory, providing a quicker and more agile way to prove several preliminary results.

On the other side, inequalities~\eqref{eqi:frac_nabla_NL_Besov} and~\eqref{eqi:frac_nabla_Besov} automatically extend the validity of the fractional Leibniz rule~\eqref{eqi:leibniz_frac_nabla} to Besov functions, solving the well-posedness issue related to the non-local part of the formula as soon as at least one of the two functions belongs to some Besov space.

Although easy, the bound~\eqref{eqi:trivial_bound_NL} is rough, since moving the absolute value inside the integral sign rules most of the cancellations originally allowed by the operator in~\eqref{eqi:def_frac_nabla_NL}. 
Having this observation in mind, one is tempted to look for a different approach exploiting the effect of the cancellations inside the integral~\eqref{eqi:def_frac_nabla_NL}.
Going back to formula~\eqref{eqi:leibniz_frac_nabla}, one can na\"ively reverse the roles of the involved operators and rewrite~\eqref{eqi:def_frac_nabla_NL} as
\begin{equation*}
\nabla^\alpha_{\NL}(f,g)
=
\nabla^\alpha(fg)
-g\,\nabla^\alpha f
-f\,\nabla^\alpha g.
\end{equation*}
At this point, one exploits the relation $\nabla^\alpha=R(-\Delta)^{\frac\alpha2}$, where $R$ is the \emph{Riesz transform} and 
$
(-\Delta)^{\frac\alpha2}
$ is the \emph{fractional Laplacian} (see \cref{subsec:notation} below for the precise definitions), so that
\begin{equation}
\label{eqi:frac_nabla_NL_rewrite}
\begin{split}
\nabla^\alpha_{\NL}(f,g)
&=
R\big(H_\alpha(f,g)\big)
+
[R,g](-\Delta)^{\frac\alpha2}f
+
[R,f](-\Delta)^{\frac\alpha2}g,
\end{split}
\end{equation} 
where 
\begin{equation}
\label{eqi:def_H_alpha}
H_\alpha(f,g)
=
(-\Delta)^{\frac\alpha2}(fg)
-g\,
(-\Delta)^{\frac\alpha2}f
-f\,
(-\Delta)^{\frac\alpha2}g
\end{equation}
and
\begin{equation}
\label{eqi:Riesz_commutator}
[R,f]g
=
R(fg)-gRf.
\end{equation} 
With this new definition at disposal, the integrability properties of the non-local fractional gradient~\eqref{eqi:def_frac_nabla_NL}  reduce to the integrability properties of the operator $H_\alpha$ in~\eqref{eqi:def_H_alpha} and the mapping properties of~$R$ and of its commutator~\eqref{eqi:Riesz_commutator}. 

On the one side, the integrability properties of the operator $H_\alpha$ date back to the fundamental work~\cite{KPV93} by Kenig--Ponce--Vega, where the authors proved that
\begin{equation}
\label{eqi:kpv}
\|H_\alpha(f,g)\|_{L^p(\R^n)}
\le
c_{n,\alpha,p}
\,
\big(
\|(-\Delta)^{\frac\alpha2}f\|_{L^p(\R^n)}
\,
\|g\|_{L^\infty(\R^n)}
+
\|f\|_{L^\infty(\R^n)}
\,
\|(-\Delta)^{\frac\alpha2}g\|_{L^p(\R^n)}
\big)
\end{equation}  
for all $f,g\in C^\infty_c(\R^n)$ and $p\in(1,+\infty)$.
The estimate in~\eqref{eqi:kpv} is often called \emph{fractional Leibniz rule} and was originally motivated by the study of the Korteweg--de Vries equation.
Although not fundamental for our purposes, it is worth to mention that the above bound can be strengthened by replacing the  $L^\infty$ norms appearing in the right-hand side with $\BMO$ seminorms, see~\cite{LS20}*{Theorem~7.1}.
Fractional Leibniz-type rules like~\eqref{eqi:kpv} have been extensively investigated in the last thirty years, see~\cites{BMMN14,BL14,CW91,d2019short,d2020correction,G12,GMN14,GK96,GO14,KP88,OW20} and the monographs~\cites{MS13,T06} for most relevant results and applications in this direction.

On the other side, the mapping properties of the commutator~\eqref{eqi:Riesz_commutator} were studied by Coifman--Rochberg--Weiss in~\cite{CRW76}, where they proved the inequality
\begin{equation}
\label{eqi:crw}
\|[R,f]g\|_{L^p(\R^n;\,\R^n)}
\le 
c_{n,p}
\,
\|f\|_{L^\infty(\R^n)}
\,
\|g\|_{L^p(\R^n)}
\end{equation}
for all $f,g\in C^\infty_c(\R^n)$ and $p\in(1,+\infty)$.
Again, although not essential for our scopes, the $L^\infty$ norm in the right-hand side of the above inequality can be replaced with the $\BMO$ seminorm, see~\cite{CRW76}*{Theorem~I} and also~\cite{LS20}*{Theorem~4.1}.

By combining~\eqref{eqi:frac_nabla_NL_rewrite} with the estimates~\eqref{eqi:kpv} and~\eqref{eqi:crw} and recalling the $L^p$-bounded\-ness of the Riesz transform for $p\in(1,+\infty)$, we conclude that
\begin{equation}
\label{eqi:frac_nabla_NL_Lp_bound}
\|\nabla^\alpha_{\NL}(f,g)\|_{L^p(\R^n;\,\R^n)}
\le
c_{n,\alpha,p}
\big(
\|\nabla^\alpha f\|_{L^p(\R^n;\,\R^n)}\,\|g\|_{L^\infty(\R^n)}
+
\|f\|_{L^\infty(\R^n)}\,\|\nabla^\alpha g\|_{L^p(\R^n;\,\R^n)}
\big)
\end{equation}
for all $f,g\in C^\infty_c(\R^n)$.
Inequality~\eqref{eqi:frac_nabla_NL_Lp_bound} allows to extend the fractional Leibniz rule~\eqref{eqi:leibniz_frac_nabla} to functions $f,g\in S^{\alpha,p}(\R^n)\cap L^\infty(\R^n)$ whenever $p\in(1,+\infty)$ by a standard approximation argument.
However, since one cannot directly pass to the limit in the expression~\eqref{eqi:def_frac_nabla_NL}, the non-local fractional gradient has to be understood in a proper weak sense via a non-local analogue of the fractional integration-by-parts formula~\eqref{eqi:ibp}, i.e.\ 
\begin{equation}
\label{eqi:ibp_NL}
\int_{\R^n}\phi\cdot\nabla^\alpha_{\NL}(f,g)\,dx
=
\int_{\R^n}f\,\div^\alpha_{\NL}(g,\phi)
\,dx
\end{equation} 
for all functions $f,g$ and vector-fields $\phi$ in suitable Besov spaces, once again thanks to the bound~\eqref{eqi:frac_nabla_NL_Besov}.

Somehow in analogy with the failure of the $L^1$-boundedness of the Riesz transform, the inequalities~\eqref{eqi:kpv} and~\eqref{eqi:crw} are not known in the boundary case \mbox{$p=1$} (and are not even expected to hold, see the discussion at the beginning of~\cite{LS20}*{Section~9}).
For this reason, we do not know if the non-local fractional gradient~\eqref{eqi:def_frac_nabla_NL} is well defined --- even in the weak sense, via~\eqref{eqi:ibp_NL} --- for $f,g\in S^{\alpha,1}(\R^n)\cap L^\infty(\R^n)$.
Thus, at the present moment, we are not able to extend the above approach to the case $p=1$.

\subsection{Main results}
We can now state the three main theorems concerning fractional Leibniz rules we are going to prove in the present paper.

Our first main result deals with the product of a $BV^{\alpha,p}$ function with a Besov function with suitable regularity depending on the exponent $p\in[1,+\infty]$. 

\begin{theorem}[Leibniz rule for $BV^{\alpha,p}$ with Besov]
\label{resi:leibniz_BV_alpha_p}
Let $\alpha\in(0,1)$ and let $p,q\in[1,+\infty]$ be such that $\frac1p+\frac1q=1$.
If $f\in BV^{\alpha,p}(\R^n)$ and
\begin{equation}
\label{eqi:cases_BV_alpha_p}
g\in
\begin{cases}
B^\alpha_{q,1}(\R^n)
&
\text{for}\
p\in\left[1,\frac n{n-\alpha}\right),
\\[2mm]
L^\infty(\R^n)\cap B^\alpha_{q,1}(\R^n)
&
\text{for}\
p\in\left[\frac n{n-\alpha},\frac1{1-\alpha}\right),
\\[2mm]
L^\infty(\R^n)\cap B^\beta_{q,1}(\R^n)\
\text{with}\ \beta\in\left(\beta_{n,p,\alpha},1\right)
&
\text{for}\
p\in\left[\frac1{1-\alpha},+\infty\right),
\\[2mm]
L^\infty(\R^n)\cap W^{\alpha,1}(\R^n)
&
\text{for}\
p=+\infty,
\end{cases}
\end{equation}
where
\begin{equation*}
\beta_{n,p,\alpha} 
= 
\begin{cases} 1 - \frac{1}{p} 
& 
\text{ if } n \ge 2\
\text{and}\
p\in\left[\frac 1{1-\alpha},\frac n{1-\alpha}\right), 
\\[4mm]
\left(1 - \frac{1}{p} \right)
\left(\alpha+\frac np\right)
& 
\text{ if } p\in\left[\frac n{1-\alpha},+\infty\right),
\end{cases}
\end{equation*}
then $fg\in BV^{\alpha,r}(\R^n)$ for all $r \in [1, p]$ with
\begin{equation*} 
D^{\alpha}(fg) 
= 
g^\star  D^{\alpha} f 
+ 
f\,
\nabla^{\alpha}g\,
\Leb{n} 
+ 
\nabla^{\alpha}_{\rm NL}(f,g)\,\Leb{n}
\quad 
\text{in}\ \M (\R^n; \R^{n}).
\end{equation*}
In addition,
\begin{equation*}
D^\alpha(fg)(\R^n)
=
\int_{\R^n}\nabla^\alpha_{\rm NL}(f,g)\,dx=0,
\end{equation*}
and 
\begin{equation*}
\int_{\R^n}
f\,\nabla^\alpha g\,dx
=
-
\int_{\R^n}
g^\star \,d D^\alpha f.
\end{equation*}
\end{theorem}

Our second main result provides an analogue of \cref{resi:leibniz_BV_alpha_p} for the product of two Bessel functions.

\begin{theorem}[Leibniz rule for $S^{\alpha,p}\cdot S^{\alpha,q}$]
\label{resi:lebniz_bessel_conjugate}
Let $\alpha\in(0,1)$ and $p,q\in(1,+\infty)$ be such that $\frac1p+\frac1q=\frac1r$ for some $r\in[1,+\infty)$.
There exists a constant $c_{n,\alpha,p,q}>0$, depending on $n$, $\alpha$, $p$ and~$q$ only, with the following property.
If $f\in S^{\alpha,p}(\R^n)$ and $g\in S^{\alpha,q}(\R^n)$, then $fg\in S^{\alpha,r}(\R^n)$ with
\begin{equation}
\label{eq:z08}	
\nabla^\alpha(fg)
=
g\,\nabla^\alpha f
+
f\,\nabla^\alpha g
+
\nabla^\alpha_{\NL}(f,g)
\quad
\text{in}\
L^r(\R^n;\R^n)
\end{equation}	
and
\begin{equation}
\label{eq:z09}
\|
\nabla^\alpha_{\NL}(f,g)
\|_{L^r(\R^n;\,\R^n)}
\le
c_{n,\alpha,p,q}
\,
\|f\|_{S^{\alpha,p}(\R^n)}
\,
\|g\|_{S^{\alpha,q}(\R^n)}.
\end{equation}
In addition, if $r=1$ then
\begin{equation}
\label{eq:z06}
\int_{\R^n}\nabla^\alpha(fg)\,dx
=
\int_{\R^n}\nabla^\alpha_{\rm NL}(f,g)\,dx=0
\end{equation}
and 
\begin{equation}
\label{eq:z07}	
\int_{\R^n}
f\,\nabla^\alpha g\,dx
=
-
\int_{\R^n}
g\,\nabla^\alpha f\,dx.
\end{equation}  
\end{theorem}

Finally, our third main result rephrases the fractional Leibniz-type formula~\eqref{eqi:kpv} by Kenig--Ponce--Vega for the product of two bounded $S^{\alpha,p}$  functions with $p\in(1,+\infty)$.
Here $\nabla^\alpha_{\NL,w}$ denotes the non-local fractional gradient defined in the weak sense via the integration-by-parts formula~\eqref{eqi:ibp_NL}, see \cref{def:weak_NL_grad} below for more details. 

\begin{theorem}
[Leibniz rule in $S^{\alpha,p}\cap L^\infty$]
\label{resi:lebniz_bessel_bounded}
Let $\alpha\in(0,1)$ and $p\in(1,+\infty)$.
There exists a constant $c_{n,\alpha,p}>0$, depending on $n$, $\alpha$ and $p$ only, with the following property.
If $f,g\in S^{\alpha,p}(\R^n)\cap L^\infty(\R^n)$, then $fg\in S^{\alpha,p}(\R^n)\cap L^\infty(\R^n)$ and $\nabla^\alpha_{\NL,w}(f,g)\in L^p(\R^n;\R^n)$, with
\begin{equation*}
\nabla^\alpha(fg)
=
g\,\nabla^\alpha f
+
f\,\nabla^\alpha g
+
\nabla^\alpha_{\NL,w}(f,g)
\quad
\text{in}\
L^p(\R^n;\R^n)
\end{equation*}	
and
\begin{equation*}
\|
\nabla^\alpha_{\NL,w}(f,g)
\|_{L^p(\R^n;\,\R^n)}
\le
c_{n,\alpha,p}
\big(
\|\nabla^\alpha f\|_{L^p(\R^n;\,\R^n)}\,\|g\|_{L^\infty(\R^n)}
+
\|f\|_{L^\infty(\R^n)}\,\|\nabla^\alpha g\|_{L^p(\R^n;\,\R^n)}	
\big).
\end{equation*}
\end{theorem}

\subsection{An application to PDEs}

In the classical setting, a second-order partial differential operator in divergence form on an open set $\Omega\subset\R^n$  can be expressed as
\begin{equation}
\label{eqi:diff_op_L}
Lu
=
-\div(A\nabla u)
+
b\cdot\nabla u
+
c\,u
\end{equation} 
for suitably regular  $u\colon\Omega\to\R$ and given $A\colon \Omega\to\R^{n^2}$, $b\colon\Omega\to\R^n$ and $c\colon\Omega\to\R$.

From the physical point of view, $u$ is the \emph{density} of some substance, so that the second-order term $-\div(A\nabla u)$ represents the \emph{diffusion} of $u$ within $\Omega$, where the coefficient matrix $A$ models the anisotropic heterogenous nature of the medium, the first order term $b\cdot\nabla u$ stands for the \emph{transport} of $u$ within $\Omega$ according to the velocity field $b$, and the zero-order term $cu$ describes the local \emph{increase} or \emph{depletion} of the substance.

From a purely differential point of view, at least for a sufficiently regular vector field $b$, there is no difference between the operator $L$ in~\eqref{eqi:diff_op_L} and the operator
\begin{equation}
\label{eqi:diff_op_L_tilde}
\tilde Lu
=
-\div(A\nabla u)
+
\div(bu)
+
\tilde c \, u,
\end{equation}
where the \emph{transport} term $b\cdot\nabla u$ has been replaced by the \emph{continuity} term $\div(bu)$. 
Indeed, by the usual Leibniz rule, one has
\begin{equation*}
\div(bu)
=
u\,\div b
+
b\cdot\nabla u,
\end{equation*}
so that we can rewrite the operator in~\eqref{eqi:diff_op_L_tilde} in the form~\eqref{eqi:diff_op_L}. 

In the present fractional framework, for $\alpha\in(0,1)$, the natural fundamental \emph{diffusion} operator is given by the \emph{fractional Laplacian} of order~$2\alpha$
\begin{equation*}
(-\Delta)^\alpha u(x)
=
c_{n,\alpha}
\int_{\R^n}\frac{u(y)-u(x)}{|y-x|^{n+2\alpha}}\,dy,
\quad
x\in\R^n,
\end{equation*} 
where $c_{n,\alpha}$ is a suitable constant.
As proved in~\cite{Sil19}*{Theorem~5.3}, one can actually write 
$(-\Delta)^\alpha=-\div^\alpha\nabla^\alpha$, so that $-\div^\alpha(A\nabla^\alpha u)$ can be seen as the natural model of an anisotropic heterogenous fractional diffusion (see also~\cite{SS15}).

Following this analogy, we may interpret $\div^\alpha(bu)$ and $b\cdot\nabla^\alpha u$ as the corresponding natural fractional continuity and transport $\alpha$-order terms, respectively.
However, differently from the integer framework, even for a smooth vector field~$b$ the fractional continuity term cannot be reinterpreted as a fractional transport term plus a zero-order term, since the fractional Leibniz rule
\begin{equation*}
\div^\alpha(bu)
=
u\,\div^\alpha b +
b\cdot\nabla^\alpha u
+
\div^\alpha_{\NL}(u,b)
\end{equation*} 
forces the presence of the non-local divergence $\alpha$-order term $\div^\alpha_{\NL}(u,b)$. 
In this sense, the most general expression of a $2\alpha$-order differential operator in (fractional) divergence form is
\begin{equation}
\label{eqi:diff_op_L_frac}
L_\alpha u
=
-\div^\alpha(A\nabla^\alpha u)
-
c_1\,\div^\alpha(b_1u)
+
b_2\cdot\nabla^\alpha u
+
c_3\,\div^\alpha_{\NL}(u,b_3)
+
c_0\,u
\end{equation}
for some given $A\colon\R^n\to\R^{n^2}$,
$b_1,b_2,b_3\colon\R^n\to\R^n$ 
and 
$c_0,c_1,c_3\colon\R^n\to\R$. 

At this point, it is natural to investigate under which (weakest possible) regularity assumptions on~$A$, $b_1$, $b_2$, $b_3$, $c_0$, $c_1$ and~$c_3$ the fractional differential operator $L_\alpha$ in~\eqref{eqi:diff_op_L_frac} is well posed. 
Imitating the usual strategy adopted in the integer setting (see~\cite{E10}*{Chapter~6} for instance), we prove the following \emph{energy estimates} for the associated bilinear form $\mathsf B_\alpha\colon S^{\alpha,2}(\R^n)\times S^{\alpha,2}(\R^n)\to\R$ given by
\begin{equation}
\label{eqi:bil_form_B}
\begin{split}
\mathsf B_\alpha[u,v]
&=
\int_{\R^n}\nabla^\alpha v\cdot A\nabla^\alpha u\,dx
+
\int_{\R^n}ub_1\cdot\nabla^\alpha(c_1v)
\,dx
+
\int_{\R^n} vb_2\cdot\nabla^\alpha u\,dx
\\
&\quad
+
\int_{\R^n}c_3v\,\div^\alpha_{\NL,w}(u,b_3)\,dx
+
\int_{\R^n}c_0uv\,dx
\end{split}
\end{equation}
for all $u,v\in S^{\alpha,2}(\R^n)$. 
In the above formula, $\div^\alpha_{\NL,w}$ denotes the non-local fractional divergence defined in the weak sense via the integration-by-parts formula~\eqref{eqi:ibp_NL}, see \cref{def:weak_NL_grad} below for more details.

\begin{proposition}[Energy estimates]
\label{res:energy_est}
Let $\alpha\in(0,1)$ be such that $\alpha<\frac n2$.
Let $A\in L^\infty(\R^n; \R^{n^2})$ be such that there exists a constant $\theta>0$ for which
\begin{equation}
\label{eq:matrix_coercive}
\xi\cdot A(x)\xi
\ge\theta|\xi|^2
\end{equation}
for a.e.\ $x\in\R^n$ and all $\xi\in\R^n$.
Let 
$b_1,b_2\in L^\infty(\R^n;\R^n)$,
$b_3\in L^\infty(\R^n;\R^n)\cap b^\alpha_{\frac n\alpha,1}(\R^n;\R^n)$,  
$c_1\in L^\infty(\R^n)\cap b^\alpha_{\frac n\alpha, 1}(\R^n)$ and $c_0, c_3 \in L^\infty(\R^n)$.
Then, there exist two constants $C,M>0$ depending on 
\begin{equation*}
\begin{split}
&
\|A\|_{L^\infty(\R^n;\,\R^{n^2})},
\
\|b_1\|_{L^\infty(\R^n;\,\R^n)},
\
\|b_2\|_{L^\infty(\R^n;\,\R^n)},
\
\|b_3\|_{L^\infty(\R^n;\,\R^n)},
\
[b_3]_{B^\alpha_{\frac n\alpha,1}(\R^n;\,\R^n)},
\\
&
\|c_1\|_{L^\infty(\R^n)},
\
[c_1]_{B^\alpha_{\frac n\alpha,1}(\R^n)},
\
\|c_0\|_{L^\infty(\R^n)},
\
\|c_3\|_{L^\infty(\R^n)}
\end{split}	
\end{equation*}
only such that 
\begin{equation}
\label{eq:energy_est_cont}
|\mathsf B_\alpha[u,v]|
\le
M
\,
\|u\|_{S^{\alpha,2}(\R^n)}
\,
\|v\|_{S^{\alpha,2}(\R^n)}
\end{equation}
and 
\begin{equation}
\label{eq:energy_est_coercive}
\frac\theta2\,\|\nabla^\alpha u\|_{L^2(\R^n; \, \R^n)}^2
\le
\mathsf B_\alpha[u,u]
+
C\,\|u\|_{L^2(\R^n)}^2
\end{equation}
for all $u,v\in S^{\alpha,2}(\R^n)$.
\end{proposition}

As a consequence, thanks to the Lax--Milgram Theorem, we obtain the following well-posedness result for the operator $L_\alpha$ on an arbitrary open set with finite measure $\Omega\subset\R^n$, generalizing~\cite{SS15}*{Theorem~1.13}. 

\begin{corollary}[Fractional boundary-value problem]
\label{res:boundary_problem}
Let $\Omega\subset\R^n$ be an open set such that $|\Omega| < +\infty$ and let
\begin{equation*}
S^{\alpha,2}_0(\Omega)
=
\set*{u\in S^{\alpha,2}(\R^n) : u=0\ \text{in}\ \R^n\setminus\Omega}.
\end{equation*}
With the same notation and the same assumptions of \cref{res:energy_est}, if $\lambda \ge C$ then,
for each $f\in L^2(\Omega)$, there exists a unique weak solution $u\in S^{\alpha,2}_0(\Omega)$ of the boundary-value problem
\begin{equation}
\label{eqi:frac_BVP}
\begin{cases}
L_\alpha u+\lambda u=f & \text{in}\ \Omega,\\[2mm]
u=0 & \text{in}\ \R^n\setminus\Omega,
\end{cases}
\end{equation} 
in the sense that 
\begin{equation*}
\mathsf B_\alpha[u,v]
+\lambda\scalar*{u,v}_{L^2(\R^n)}= \scalar*{\tilde f, v}_{L^2(\R^n)}
\end{equation*}
for all $v \in S^{\alpha, 2}_0(\Omega)$, where $\tilde f$ is the extension-by-zero of $f$ outside~$\Omega$ and $\scalar*{\cdot, \cdot}_{L^2(\R^n)}$ is the standard scalar product in $L^2(\R^n)$.
\end{corollary}

Note that \cref{res:boundary_problem} actually holds for any non-empty set $\Omega$ with finite measure, not necessarily open---we added this topological assumption in analogy with the classical setting.

Although we do not pursue this direction in detail here, it is worth mentioning that the approach leading to \cref{res:boundary_problem} can be likewise applied to the boundary-value problem for the parabolic fractional operator $\partial_t+L_\alpha$ (where now the coefficients of $L_\alpha$ may depend also on the time variable $t\ge0$) by suitably adapting the \emph{Garlekin's approximation method} (see~\cite{E10}*{Chapter~7} for example) to the present fractional setting.

The validity of \cref{res:energy_est} follows from the following Leibniz rule for the product of an $S^{\alpha,p}$ function with a bounded function having finite $B^\alpha_{r,1}$-seminorm, where $p\in(1,+\infty)$, $n>\alpha p$ and $r \in \left [ \frac{n}{\alpha}, + \infty \right ]$.
The proof of \cref{res:Leibniz_Bessel_bounded_Besov} combines the previous results with the fractional Sobolev embedding $S^{\alpha,p}(\R^n)\subset L^{\frac{np}{n-\alpha p}}(\R^n)$.
Here $\nabla^\alpha_{\NL,w}$ denotes the non-local fractional gradient defined in the weak sense via the integration-by-parts formula~\eqref{eqi:ibp_NL}, see \cref{def:weak_NL_grad} below for more details. 

\begin{theorem}[Leibniz rule for $S^{\alpha,p}$ with bounded Besov]
\label{res:Leibniz_Bessel_bounded_Besov}
Let $\alpha\in(0,1)$ and $p\in(1,+\infty)$ be such that $n>\alpha p$, and let $r \in \left [ \frac{n}{\alpha}, + \infty \right ]$.
There exists a constant $c_{n,\alpha,p}>0$, depending on~$n$, $\alpha$ and~$p$ only, with the following property. 
If $f\in S^{\alpha,p}(\R^n)$ and $g\in L^\infty(\R^n)\cap b^\alpha_{r,1}(\R^n)$, then $fg\in S^{\alpha,p}(\R^n)$ with
\begin{equation}
\label{eq:Leibniz_Bessel_bounded_Besov}
\nabla^\alpha(fg)
=
g\,\nabla^\alpha f
+
f\,\nabla^\alpha g
+
\nabla^\alpha_{\NL,w}(f,g)
\quad
\text{in}\ L^p(\R^n;\R^n)
\end{equation}
and
\begin{equation}
\label{eq:Leibniz_Bessel_bounded_Besov_est}
\|f\,\nabla^\alpha g\|_{L^p(\R^n;\,\R^n)}
+
\|\nabla^\alpha_{\NL,w}(f,g)\|_{L^p(\R^n;\,\R^n)}
\le
c_{n,\alpha,p}^{\frac n{\alpha r}}
\, \|f\|_{L^p(\R^n)}^{1-\frac n{\alpha r}}
\|\nabla^\alpha f\|_{L^p(\R^n;\,\R^n)}^{\frac n{\alpha r}}
\, 
[g]_{B^\alpha_{r,1}(\R^n)}.
\end{equation}
In particular, if $r = \frac{n}{\alpha}$, then
\begin{equation}
\label{eq:Leibniz_Bessel_bounded_Besov_est_bis}
\|f\,\nabla^\alpha g\|_{L^p(\R^n;\,\R^n)}
+
\|\nabla^\alpha_{\NL,w}(f,g)\|_{L^p(\R^n;\,\R^n)}
\le
c_{n,\alpha,p}
\,
\|\nabla^\alpha f\|_{L^p(\R^n;\,\R^n)}
\, 
[g]_{B^\alpha_{\frac{n}{\alpha},1}(\R^n)}.
\end{equation}
\end{theorem}

It is worth noting that an analogue of \cref{res:Leibniz_Bessel_bounded_Besov} can be stated for $BV^{\alpha,p}$ functions, with $p\in\left[1,\frac n{n-\alpha}\right)$.
We refer the interested reader to \cref{res:Leibniz_frac_BV_bounded_Besov} below for the precise statement.

\subsection{Organization of the paper}

The paper is organized as follows.

In \cref{sec:preliminaries}, we recall the main notation used throughout the paper and we prove several basic results for the involved fractional differential operators in the framework of Besov spaces.

In \cref{sec:leibniz_BV_alpha_p}, we establish our first main result \cref{resi:leibniz_BV_alpha_p}, proving each case in~\eqref{eqi:cases_BV_alpha_p} separately, see \cref{res:BV_alpha_p_leibniz_C_b},
\cref{res:giovannino},
\cref{res:BV_alpha_p_leibniz_Sob_embedding}, 
\cref{res:remove_C_p_finite} and 
\cref{res:remove_C_p_infty}, 
respectively.
We also discuss some strictly related Gauss--Green formulas, see \cref{subsec:GG_formulas}.

In \cref{sec:leibniz_bessel}, after having settled inequality~\eqref{eqi:frac_nabla_NL_Lp_bound}, we show  \cref{resi:lebniz_bessel_conjugate} and \cref{resi:lebniz_bessel_bounded}.

Finally, in \cref{sec:application}, we first prove \cref{res:Leibniz_Bessel_bounded_Besov} and then we conclude our paper by establishing \cref{res:energy_est} and \cref{res:boundary_problem}.

\section{Preliminaries}
\label{sec:preliminaries}

\subsection{General notation}
\label{subsec:notation}

We start with a brief description of the main notation used in this paper. In order to keep the exposition the most reader-friendly as possible, we retain the same notation adopted in the previous works~\cites{CS19,CS19-2,BCCS20,CSS21}.

Given an open set $\Omega\subset\R^n$, we say that a set $E$ is compactly contained in $\Omega$, and we write \mbox{$E\Subset\Omega$}, if the $\closure{E}$ is compact and contained in $\Omega$.
We let $\Leb{n}$ and $\Haus{\alpha}$ be the $n$-dimensional Lebesgue measure and the $\alpha$-dimensional Hausdorff measure on $\R^n$, respectively, with $\alpha\in[0,n]$. 
Unless otherwise stated, a measurable set is a $\Leb{n}$-measurable set. 
We also use the notation $|E|=\Leb{n}(E)$. 
All functions we consider in this paper are Lebesgue measurable, unless otherwise stated. 
We denote by $B_r(x)$ the standard open Euclidean ball with center $x\in\R^n$ and radius $r>0$. 
We let $B_r=B_r(0)$. 
Recall that $\omega_{n} = |B_1|=\pi^{\frac{n}{2}}/\Gamma\left(\frac{n+2}{2}\right)$ and $\Haus{n-1}(\partial B_{1}) = n \omega_n$, where $\Gamma$ is Euler's \emph{Gamma function}.

For $k \in \N_{0} \cup \set{+ \infty}$ and $m \in \N$, we let $C^{k}_{c}(\Omega ; \R^{m})$ and $\Lip_c(\Omega; \R^{m})$ be the spaces of $C^{k}$-regular and, respectively, Lipschitz-regular, $m$-vector-valued functions defined on~$\R^n$ with compact support in the open set~$\Omega\subset\R^n$. 
Analogously, we let $C^{k}_{b}(\Omega ; \R^{m})$ and $\Lip_b(\Omega; \R^{m})$ be the spaces of $C^{k}$-regular and, respectively, Lipschitz-regular, $m$-vector-valued bounded functions defined on the open set $\Omega\subset\R^n$. 
In the case $k = 0$, we drop the superscript and simply write $C_{c}(\Omega ; \R^{m})$ and $C_{b}(\Omega ; \R^{m})$.

For $m\in\N$, the total variation on~$\Omega$ of the $m$-vector-valued Radon measure $\mu$ is defined as
\begin{equation*}
|\mu|(\Omega)
=
\sup\set*{\int_\Omega\phi\cdot d\mu : \phi\in C^\infty_c(\Omega;\R^m),\ \|\phi\|_{L^\infty(\Omega;\,\R^m)}\le1}.
\end{equation*}
We thus let $\M(\Omega;\R^m)$ be the space of $m$-vector-valued Radon measure  with finite total variation on $\Omega$.
We say that $(\mu_k)_{k\in\N}\subset\M(\Omega;\R^m)$ \emph{weakly converges} to $\mu\in\M (\Omega;\R^m)$, and we write $\mu_k\weakto\mu$ in $\M (\Omega;\R^m)$ as $k\to+\infty$, if 
\begin{equation}\label{eq:def_weak_conv_meas}
\lim_{k\to+\infty}\int_\Omega\phi\cdot d\mu_k=\int_\Omega\phi\cdot d\mu
\end{equation} 
for all $\phi\in C_c(\Omega;\R^m)$. Note that we make a little abuse of terminology, since the limit in~\eqref{eq:def_weak_conv_meas} actually defines the \emph{weak*-convergence} in~$\M (\Omega;\R^m)$. 

For any exponent $p\in[1,+\infty]$, we let $L^p(\Omega;\R^m)$ be the space of $m$-vector-valued Lebesgue $p$-integrable functions on~$\Omega$.

We let
\begin{equation*}
W^{1,p}(\Omega;\R^m)
=
\set*{u\in L^p(\Omega;\R^m) : [u]_{W^{1,p}(\Omega;\,\R^m)}=\|\nabla u\|_{L^p(\Omega;\,\R^{nm})}<+\infty}
\end{equation*}
be the space of $m$-vector-valued Sobolev functions on~$\Omega$, for instance see~\cite{Leoni17}*{Chapter~11} for its precise definition and main properties. 
We let
\begin{equation*}
BV(\Omega;\R^m)
=
\set*{u\in L^1(\Omega;\R^m) : [u]_{BV(\Omega;\,\R^m)}=|Du|(\Omega)<+\infty}
\end{equation*}
be the space of $m$-vector-valued functions of bounded variation on~$\Omega$, see for instance~\cite{AFP00}*{Chapter~3} or~\cite{EG15}*{Chapter~5} for its precise definition and main properties. 
For $\alpha\in(0,1)$ and $p\in[1,+\infty)$, we let
\begin{equation*}
W^{\alpha,p}(\Omega;\R^m)
=\set*{u\in L^p(\Omega;\R^m) : [u]_{W^{\alpha,p}(\Omega;\,\R^m)}=\left(\int_\Omega\int_\Omega\frac{|u(x)-u(y)|^p}{|x-y|^{n+p\alpha}}\,dx\,dy\right)^{\frac{1}{p}}\!<+\infty}
\end{equation*}
be the space of $m$-vector-valued fractional Sobolev functions on~$\Omega$, see~\cite{DiNPV12} for its precise definition and main properties. 
For $\alpha\in(0,1)$ and $p=+\infty$, we simply let
\begin{equation*}
W^{\alpha,\infty}(\Omega;\R^m)=\set*{u\in L^\infty(\Omega;\R^m) : \sup_{x,y\in \Omega,\, x\neq y}\frac{|u(x)-u(y)|}{|x-y|^\alpha}<+\infty},
\end{equation*}
so that $W^{\alpha,\infty}(\Omega;\R^m)=C^{0,\alpha}_b(\Omega;\R^m)$, the space of $m$-vector-valued bounded $\alpha$-H\"older continuous functions on~$\Omega$.

For $\alpha\in(0,1)$ and $p,q\in[1,+\infty]$, we let
\begin{equation*}
B^\alpha_{p,q}(\R^n;\R^m)
=
\set*{
u\in L^p(\R^n;\R^m)
:
[u]_{B^\alpha_{p,q}(\R^n;\,\R^m)}
<
+\infty
}
\end{equation*}
be the space of $m$-vector-valued Besov functions on $\R^n$, see~\cite{Leoni17}*{Chapter~17} for its precise definitions and main properties, where 
\begin{equation*}
[u]_{B^\alpha_{p,q}(\R^n;\,\R^m)}
=
\begin{cases}
\left(
\displaystyle\int_{\R^n}
\frac{\|u(\cdot+h)-u\|_{L^p(\R^n;\,\R^m)}^q}{|h|^{n+q\alpha}}
\,dh
\right)^{\frac1q}
&
\text{if}\
q\in[1,+\infty),
\\[10mm]
\sup\limits_{h\in\R^n\setminus\set*{0}}
\dfrac{\|u(\cdot+h)-u\|_{L^p(\R^n;\,\R^m)}}{|h|^{\alpha}}
&
\text{if}\
q=\infty.
\end{cases}
\end{equation*} 
We also let
\begin{equation*}
b^\alpha_{p,q}(\R^n;\R^m)
=
\set*{
u\in L^1_{\loc}(\R^n;\R^m)
:
[u]_{B^\alpha_{p,q}(\R^n;\,\R^m)}
<
+\infty
}.
\end{equation*}

In order to avoid heavy notation, if the elements of a function space $F(\Omega;\R^m)$ are real-valued (i.e., $m=1$), then we will drop the target space and simply write~$F(\Omega)$.


Given $\alpha\in(0,1)$, we also let
\begin{equation}\label{eq:def_frac_Laplacian}
(- \Delta)^{\frac{\alpha}{2}} f(x) 
= 
\nu_{n, \alpha}
\int_{\R^n} \frac{f(x + y) - f(x)}{|y|^{n + \alpha}}\,dy,
\quad
x\in\R^n,
\end{equation}
be the fractional Laplacian (of order~$\alpha$) of $f \in\Lip_b(\R^{n};\R^m)$, where
\begin{equation*}
\nu_{n,\alpha}=2^\alpha\pi^{-\frac n2}\frac{\Gamma\left(\frac{n+\alpha}{2}\right)}{\Gamma\left(-\frac{\alpha}{2}\right)},
\quad
\alpha\in(0,1).
\end{equation*}

Finally, we let
\begin{equation}\label{eq:def_Riesz_transform}
R f(x)
=
\pi^{-\frac{n+1}2}\,\Gamma\left(\tfrac{n+1}{2}\right)\,\lim_{\eps\to0^+}\int_{\set*{|y|>\eps}}\frac{y\,f(x+y)}{|y|^{n+1}}\,dy,
\quad
x\in\R^n,
\end{equation}
be the (vector-valued) \emph{Riesz transform} of a (sufficiently regular) function~$f$. 
We refer the reader to~\cite{G14-M}*{Sections~2.1 and~2.4.4}, \cite{S70}*{Chapter~III, Section~1} and~\cite{S93}*{Chapter~III} for a more detailed exposition. 
We warn the reader that the definition in~\eqref{eq:def_Riesz_transform} agrees with the one in~\cites{S93} and differs from the one in~\cites{G14-M,S70} for a minus sign.
The Riesz transform~\eqref{eq:def_Riesz_transform} is a singular integral of convolution type, thus in particular it defines a continuous operator $R\colon L^p(\R^n)\to L^p(\R^n;\R^{n})$ for any given $p\in(1,+\infty)$, see~\cite{G14-C}*{Corollary~5.2.8}.
We also recall that its components $R_i$ satisfy
\begin{equation*}
\sum_{i=1}^nR_i^2=-\mathrm{Id}
\quad
\text{on}\ L^2(\R^n),
\end{equation*}
see~\cite{G14-C}*{Proposition~5.1.16}.

\subsection{The operators \texorpdfstring{$\nabla^\alpha$}{nablaˆalpha} and \texorpdfstring{$\div^\alpha$}{divˆalpha} and the associated fractional spaces} 

We are going to briefly revise the theory on the non-local operators~$\nabla^\alpha$ and~$\div^\alpha$, see~\cites{Sil19,CS19,CS19-2,BCCS20,CSS21} and~\cite{P16}*{Section~15.2}.

Let $\alpha\in(0,1)$ and set 
\begin{equation*}
\mu_{n, \alpha} 
= 
2^{\alpha}\, \pi^{- \frac{n}{2}}\, \frac{\Gamma\left ( \frac{n + \alpha + 1}{2} \right )}{\Gamma\left ( \frac{1 - \alpha}{2} \right )}.
\end{equation*}
Recall that fractional operators $\nabla^\alpha$ and $\div^\alpha$ are defined by setting
\begin{equation}
\label{eq:def_nabla_alpha}
\nabla^\alpha f(x)
= 
\mu_{n,\alpha}
\int_{\R^n}\frac{(y-x)(f(y)-f(x))}{|y-x|^{n+\alpha+1}}\,dy,
\quad
x\in\R^n,
\end{equation}
for $f\in\Lip_c(\R^n)$ and
\begin{equation}
\label{eq:def_div_alpha}
\div^\alpha \phi(x)
=
\mu_{n,\alpha}
\int_{\R^n}\frac{(y-x)\cdot(\phi(y)-\phi(x))}{|y-x|^{n+\alpha+1}}\,dy,
\quad
x\in\R^n,
\end{equation}
for $\phi\in\Lip_c(\R^n;\R^n)$.
Thanks to~\cite{CS19}*{Corollary~2.3}, we know that 
\begin{equation*}
\nabla^\alpha
\colon
\Lip_c(\R^n)
\to
L^1(\R^n;\R^n)
\cap
L^\infty(\R^n;\R^n)
\end{equation*}
and, analogously, 
\begin{equation*}
\div^\alpha
\colon
\Lip_c(\R^n;\R^n)
\to
L^1(\R^n)
\cap
L^\infty(\R^n),
\end{equation*}
are two linear continuous operators. 
In addition, the fractional operators $\nabla^\alpha$ and $\div^\alpha$ are \emph{dual}, in the sense that
\begin{equation}
\label{eq:ibp_Besov}
\int_{\R^n} f\,\div^\alpha\phi\,dx
=-
\int_{\R^n}\phi\cdot\nabla^\alpha f\,dx,
\end{equation}
for all $f\in\Lip_c(\R^n)$ and $\phi\in\Lip_c(\R^n;\R^n)$, see~\cite{Sil19}*{Section~6} and~\cite{CS19}*{Lemma~2.5}.

We can now quickly recall the functional spaces naturally induced by the two fractional operators~\eqref{eq:def_nabla_alpha} and~\eqref{eq:def_div_alpha}, see~\cites{CS19,CS19-2,CSS21} for a detailed exposition.
 
Let $p\in[1,+\infty]$.
We say that a function $f\in L^p(\R^n)$ belongs to the space $BV^{\alpha,p}(\R^n)$ if
\begin{equation}
\label{eq:def_frac_var_p}
|D^\alpha f|(\R^n)
=
\sup\set*{\int_{\R^n}f\,\div^\alpha\phi\,dx : \phi\in C^\infty_c(\R^n;\R^n),\ \|\phi\|_{L^\infty(\R^n;\,\R^n)}\le1}
<+\infty.
\end{equation}	
In the case $p=1$, we simply write $BV^{\alpha,1}(\R^n)=BV^\alpha(\R^n)$.
The resulting linear space 
\begin{equation*}
BV^{\alpha,p}(\R^n)=
\set*{f\in L^p(\R^n) : |D^\alpha f|(\R^n)<+\infty}
\end{equation*}
endowed with the norm
\begin{equation*}
\|f\|_{BV^{\alpha,p}(\R^n)}
=
\|f\|_{L^p(\R^n)}+|D^\alpha f|(\R^n),
\quad
f\in BV^{\alpha,p}(\R^n),
\end{equation*}
is a Banach space and the fractional variation defined in~\eqref{eq:def_frac_var_p} is lower semicontinuous with respect to the $L^p$-convergence. 
Arguing as in the proof of~\cite{CS19}*{Theorem 3.2}, it is possible to show that a function $f \in L^p(\R^n)$ belongs to $BV^{\alpha, p}(\R^n)$ if and only if there exists a finite vector valued Radon measure $D^{\alpha} f$ satisfying 
\begin{equation*}
\int_{\R^n} f \, \div^{\alpha} \phi \, dx = - \int_{\R^n} \phi \cdot d D^{\alpha} f
\end{equation*}
for all $\phi \in C^{\infty}_c(\R^n; \R^n)$, see \cite{CSS21}*{Theorem 3.1}. 
As in~\cite{CSS21}*{Proposition 3.7}, the vector fields~$\phi$ can actually be taken in a wider class depending on the value of $p$.

In a similar fashion, we say that a function $f\in L^p(\R^n)$ belongs to the space $S^{\alpha,p}(\R^n)$ if there exists a function $\nabla^\alpha f\in L^p(\R^n;\R^n)$, called \emph{weak $\alpha$-fractional gradient} of~$f$, satisfying~\eqref{eq:ibp_Besov} for all $\phi\in C^\infty_c(\R^n;\R^n)$. 
We point out that, in general, if $f \in S^{\alpha,p}(\R^n)$, then $\nabla^\alpha f$ is not necessarily given by \eqref{eq:def_nabla_alpha}.
The resulting linear space
\begin{equation*}
S^{\alpha,p}(\R^n)
=
\set*{f\in L^p(\R^n): \nabla^\alpha f\in L^p(\R^n;\R^n)}
\end{equation*}
endowed with the norm
\begin{equation*}
\|f\|_{S^{\alpha,p}(\R^n)}
=
\|f\|_{L^p(\R^n)}
+
\|\nabla^\alpha f\|_{L^p(\R^n;\,\R^n)},
\quad
f\in S^{\alpha,p}(\R^n),
\end{equation*}  
is a Banach space. 
In addition, $S^{\alpha,p}(\R^n)$ can be identified with the Bessel potential space $L^{\alpha,p}(\R^n)$ whenever $p\in(1,+\infty)$, see~\cites{BCCS20,KS21}. 

\subsection{The operators \texorpdfstring{$\mathcal D^\alpha$}{Dˆalpha} and \texorpdfstring{$\mathcal D^\alpha_{\rm NL}$}{Dˆalpha-NL} on Besov spaces}

We are now going to study the fractional operators~\eqref{eq:def_nabla_alpha} ad~\eqref{eq:def_div_alpha} on Besov spaces.
To this aim, we need to introduce two following auxiliary operators. 

Let $\alpha\in(0,1)$. 
We let
\begin{equation*}
\mathcal D^\alpha f(x)
=
\int_{\R^n}\frac{|f(x+h)-f(x)|}{|h|^{n+\alpha}}\,dh,
\quad
x\in\R^n,
\end{equation*}
for all $f\in\Lip_c(\R^n;\R^m)$, $m\in\N$.
Analogously, we let 
\begin{equation}
\label{eq:def_D_alpha_NL}
\mathcal D_{\rm NL}^\alpha(f,g)(x)
=
\int_{\R^n}
\frac{|f(x+h)-f(x)|\,|g(x+h)-g(x)|}{|h|^{n+\alpha}}
\,dh,
\quad x\in\R^n,
\end{equation}
for all $f,g\in\Lip_c(\R^n;\R^m)$, $m\in\N$.
The following result proves that the operators~$\mathcal D^\alpha$ and~$\mathcal D^\alpha_{\NL}$ naturally extend to Besov spaces.

\begin{lemma}[Integrability of $\mathcal D^\alpha$ and $\mathcal D_{\rm NL}^\alpha$]
\label{res:D_alpha_D_alpha_NL_int}
Let $\alpha,\beta,\gamma\in(0,1)$ be such that $\alpha=\beta+\gamma$.
Let $p,q,r,s,t\in[1,+\infty]$ be such that $\frac1p+\frac1q=\frac1r$ and $\frac1s+\frac1t=1$.

\begin{enumerate}[(i)]

\item
\label{item:D_alpha_int}
The operator
$\mathcal D^\alpha$ satisfies
\begin{equation}
\label{eq:D_alpha_int}
\|\mathcal D^\alpha f\|_{L^p(\R^n)}
\le
[f]_{B^\alpha_{p,1}(\R^n)}
\end{equation}
for all $f\in b^\alpha_{p,1}(\R^n)$.

\item
\label{item:D_alpha_NL_int} 
The operator
$\mathcal D^\alpha_{\rm NL}$ satisfies
\begin{equation}
\label{eq:D_alpha_NL_int}
\|
\mathcal D^\alpha_{\NL}(f,g)
\|_{L^r(\R^n)}
\le
\begin{cases}
[f]_{B^{\beta}_{p,s}(\R^n)}
\,
[g]_{B^{\gamma}_{q,t}(\R^n)}
&
\text{
for $f\in b^{\beta}_{p,s}(\R^n)$, $g\in b^{\gamma}_{q,t}(\R^n)$,
}
\\[2mm]
2
\,
\|f\|_{L^p(\R^n)}
\,
[g]_{B^\alpha_{q,1}(\R^n)}	
&
\text{
for $f\in L^p(\R^n)$, 
$g\in b^{\alpha}_{q,1}(\R^n)$,
}
\\[2mm]
2
\,
[f]_{B^\alpha_{p,1}(\R^n)}	
\,
\|g\|_{L^q(\R^n)}
&
\text{
for 
$f\in b^{\alpha}_{p,1}(\R^n)$,
$g\in L^q(\R^n)$.
}
\end{cases}
\end{equation}
\end{enumerate}
\end{lemma}

\begin{proof}
We prove the two statements separately.

\smallskip

\textit{Proof of~\eqref{item:D_alpha_int} }
By Minkowski's integral inequality, we can estimate
\begin{align*}
\|\mathcal D^\alpha f\|_{L^p(\R^n)}
&=
\bigg\|
\int_{\R^n}\frac{|f(\cdot+h)-f|}{|h|^{n+\alpha}}\,dh
\,
\bigg\|_{L^p(\R^n)}
\\
&\le
\int_{\R^n}\frac{\|f(\cdot+h)-f\|_{L^p(\R^n)}}{|h|^{n+\alpha}}\,dh
\\
&=
[f]_{B^\alpha_{p,1}(\R^n)}
\end{align*}	
and~\eqref{eq:D_alpha_int} follows.

\smallskip

\textit{Proof of~\eqref{item:D_alpha_NL_int}}.
Let us assume that $s,t\in(1,+\infty)$.
By Minkowski's integral inequality and H\"older's (generalized) inequality, we can estimate
\begin{align*}
\|
\mathcal D^\alpha_{\NL}(f,g)
\|_{L^r(\R^n)}
&\le
\int_{\R^n}
\frac{\||f(\cdot+h)-f|\,|g(\cdot+h)-g|\|_{L^r(\R^n)}}{|h|^{n+\alpha}}
\,dh
\\
&\le
\int_{\R^n}
\frac{\|f(\cdot+h)-f\|_{L^p(\R^n)}\,\|g(\cdot+h)-g\|_{L^q(\R^n)}}{|h|^{n+\alpha}}
\,dh
\\
&\le
\int_{\R^n}
\frac{\|f(\cdot+h)-f\|_{L^p(\R^n)}}{|h|^{\frac{n}{s}+\beta}}
\,
\frac{\|g(\cdot+h)-g\|_{L^q(\R^n)}}{|h|^{\frac{n}{t}+\gamma}}
\,dh
\\
&\le
\left(
\int_{\R^n}
\frac{\|f(\cdot+h)-f\|_{L^p(\R^n)}^s}{|h|^{n+s\beta}}
\,dh
\right)^{\frac1s}
\left(
\int_{\R^n}
\frac{\|g(\cdot+h)-g\|_{L^p(\R^n)}^t}{|h|^{n+t\gamma}}
\,dh
\right)^{\frac1t}
\\
&=
[f]_{B^{\beta}_{p,s}(\R^n)}
\,
[g]_{B^{\gamma}_{q,t}(\R^n)}
\end{align*}
and~\eqref{eq:D_alpha_NL_int} follows.
The remaining cases are similar and are left to the reader. 
\end{proof}

In the following result, we prove a Leibinz-type rule for the operator $\mathcal D^\alpha$ for the product of two Besov functions.

\begin{lemma}[Leibniz-type rule for $\mathcal D^\alpha$]
\label{res:D_alpha_product}
Let $\alpha\in(0,1)$ and let $p,q,r\in[1,+\infty]$ be such that $\frac1p+\frac1q=\frac1r$.
If $f\in B^\alpha_{p,1}(\R^n)$ and $g\in B^\alpha_{q,1}(\R^n)$, then $\mathcal D^\alpha(fg)\in L^r(\R^n)$ with
\begin{equation*}
\|\mathcal D^\alpha(fg)\|_{L^r(\R^n)}
\le
[f]_{B^\alpha_{p,1}(\R^n)}\|g\|_{L^q(\R^n)}
+
\|f\|_{L^p(\R^n)}
[g]_{B^\alpha_{q,1}(\R^n)}	
\end{equation*}
\end{lemma}

\begin{proof}
For $x\in\R^n$, we can estimate
\begin{align*}
\mathcal D^\alpha(fg)(x)
&=
\int_{\R^n}
\frac{|f(y)g(y)-f(x)g(x)|}{|y-x|^{n+\alpha}}
\,dy
\\
&\le
\int_{\R^n}
\frac{|g(y)|\,|f(y)-f(x)|+|f(x)|\,|g(y)-g(x)|}{|y-x|^{n+\alpha}}
\,dy
\\
&=
\int_{\R^n}
\frac{|g(x+h)|\,|f(x+h)-f(x)|}{|h|^{n+\alpha}}
\,dh
+
|f(x)|\,\mathcal D^\alpha g(x).
\end{align*}
Therefore, by Minkowski's integral inequality and H\"older's (generalized) inequality, we get
\begin{align*}
\|\mathcal D^\alpha(fg)\|_{L^r(\R^n)}
&\le
\int_{\R^n}
\frac{\||g(\cdot+h)|\,|f(\cdot+h)-f|\|_{L^r(\R^n)}}{|h|^{n+\alpha}}
\,dh
+
\|f\,\mathcal D^\alpha g\|_{L^r(\R^n)}
\\
&\le
\int_{\R^n}
\frac{\|g(\cdot+h)\|_{L^q(\R^n)}\,\|f(\cdot+h)-f\|_{L^p(\R^n)}}{|h|^{n+\alpha}}
\,dh
+
\|f\|_{L^p(\R^n)}
\|\mathcal D^\alpha g\|_{L^q(\R^n)}
\\
&\le
[f]_{B^\alpha_{p,1}(\R^n)}\|g\|_{L^q(\R^n)}
+
\|f\|_{L^p(\R^n)}
[g]_{B^\alpha_{q,1}(\R^n)}
\end{align*}
thanks to \cref{res:D_alpha_D_alpha_NL_int}\eqref{item:D_alpha_int} and the conclusion readily follows.
\end{proof}

\subsection{The operators \texorpdfstring{$\nabla^\alpha$}{nablaˆalpha} and \texorpdfstring{$\div^\alpha$}{divˆalpha} on Besov spaces} 

We are now ready to study the properties of the fractional operators~\eqref{eq:def_nabla_alpha} and~\eqref{eq:def_div_alpha} on Besov spaces.
 
We start by observing that from \cref{res:D_alpha_D_alpha_NL_int}\eqref{item:D_alpha_int} we can easily deduce the following integrability result for the operators $\nabla^\alpha$ and $\div^\alpha$ when applied to Besov functions.

\begin{corollary}[Integrability of $\nabla^\alpha$ and $\div^\alpha$]
\label{res:nabla_alpha_div_alpha_in_L_p}
Let $\alpha\in(0,1)$ and $p\in[1,+\infty]$.
The linear operators in~\eqref{eq:def_nabla_alpha} and~\eqref{eq:def_div_alpha} can be extended to two linear operators (for which we retain the same notation) satisfying 
\begin{equation*}
\begin{split}
\|\nabla^\alpha f\|_{L^p(\R^n;\,\R^n)}
&\le
\mu_{n,\alpha}\,
[f]_{B^\alpha_{p,1}(\R^n)}
\\
\|\div^\alpha \phi\|_{L^p(\R^n)}
&\le
\mu_{n,\alpha}\,
[\phi]_{B^\alpha_{p,1}(\R^n;\,\R^n)}
\end{split}
\end{equation*}
for all $f\in b^\alpha_{p,1}(\R^n)$ and $\phi\in b^\alpha_{p,1}(\R^n;\R^n)$.
\end{corollary}

In addition, the Leibniz-type rule for the operator $\mathcal D^\alpha$ proved in \cref{res:D_alpha_product} allows us to deduce the following integrability property for $\nabla^\alpha$ and $\div^\alpha$ when applied to product of Besov functions.

\begin{corollary}
[Integrability of $\nabla^\alpha$ and $\div^\alpha$ for products]
\label{res:nabla_alpha_div_alpha_product}
Let $\alpha\in(0,1)$ and let $p,q,r\in[1,+\infty]$ be such that $\frac1p+\frac1q=\frac1r$.
\begin{enumerate}[(i)]

\item
\label{item:nabla_alpha_product}
If $f\in B^\alpha_{p,1}(\R^n)$ and $g\in B^\alpha_{q,1}(\R^n)$, then $\nabla^\alpha(fg)\in L^r(\R^n;\R^n)$.
In addition, if $r=1$ then 
\begin{equation}
\label{eq:nabla_alpha_zero_mean}
\int_{\R^n}\nabla^\alpha (fg)\,dx=0.
\end{equation}

\item
\label{item:div_alpha_product}
If $f\in B^\alpha_{p,1}(\R^n)$ and $\phi\in B^\alpha_{q,1}(\R^n;\R^n)$, then $\div^\alpha(f\phi)\in L^r(\R^n)$.
In addition, if $r=1$ then 
\begin{equation}
\label{eq:div_alpha_zero_mean}
\int_{\R^n}\div^\alpha (f\phi)\,dx=0.
\end{equation}
\end{enumerate}
\end{corollary}

In the proof of the above \cref{res:nabla_alpha_div_alpha_product}, we need the following consequence of~\cite{CS19-2}*{Proposition~2.7} whose simple proof is left to the reader.

\begin{lemma}[Zero total measure]
\label{res:zero_total_mesure}
Let $\alpha\in(0,1)$. 
If $f\in BV^\alpha(\R^n)$, then $D^\alpha f(\R^n)=0$.
\end{lemma}

\begin{proof}[Proof of \cref{res:nabla_alpha_div_alpha_product}]
Thanks to \cref{res:D_alpha_product}, we just need to prove~\eqref{eq:nabla_alpha_zero_mean} and~\eqref{eq:div_alpha_zero_mean} in the case $r=1$.
To this aim, let us set $u=fg\in L^1(\R^n)$.
By \cref{res:D_alpha_product}, we know that $\mathcal D^\alpha u\in L^1(\R^n)$. 
Consequently, $u\in S^{\alpha,1}(\R^n)$ by Fubini's Theorem and so~\eqref{eq:nabla_alpha_zero_mean} follows from \cref{res:zero_total_mesure}, since $S^{\alpha,1}(\R^n) \subset BV^{\alpha}(\R^n)$, thanks to \cite{CS19}*{Theorem 3.25}. 
The proof of~\eqref{eq:div_alpha_zero_mean} is similar and is left to the reader.
\end{proof}

Thanks to \cref{res:D_alpha_D_alpha_NL_int}\eqref{item:D_alpha_int} again, we can easily extend the validity of the integration-by-parts formula~\eqref{eq:ibp_Besov} to Besov functions.

\begin{lemma}[Duality for $\nabla^\alpha$ and $\div^\alpha$]
\label{res:nabla_div_ibp_Besov}
Let $\alpha\in(0,1)$ and let $p,q\in[1,+\infty]$ be such that $\frac1p+\frac1q=1$.
If $f\in B^\alpha_{p,1}(\R^n)$ and $\phi\in B^\alpha_{q,1}(\R^n;\R^n)$, then~\eqref{eq:ibp_Besov} holds. As a consequence, $B^\alpha_{p,1}(\R^n)\subset S^{\alpha,p}(\R^n)$ with continuous inclusion.
\end{lemma}

\begin{proof}
By \cref{res:nabla_alpha_div_alpha_in_L_p}, both sides of~\eqref{eq:ibp_Besov} are well posed. 
Thanks to \cref{res:D_alpha_D_alpha_NL_int}\eqref{item:D_alpha_int} and Fubini's Theorem, we can write
\begin{align*}
\int_{\R^n} f\,\div^\alpha\phi\,dx
&=
\mu_{n,\alpha}
\int_{\R^n}
\int_{\R^n}
f(x)\,
\frac{(y-x)\cdot(\phi(y)-\phi(x))}{|y-x|^{n+\alpha+1}}
\,dy	
\,dx
\\
&=
\mu_{n,\alpha}
\int_{\R^n}
\int_{\R^n}
f(x)\,
\frac{(x-y)\cdot(\phi(x)-\phi(y))}{|x-y|^{n+\alpha+1}}
\,dx	
\,dy
\end{align*}
and, similarly,
\begin{align*}
\int_{\R^n}\phi\cdot\nabla^\alpha f\,dy
=
\mu_{n,\alpha}
\int_{\R^n}
\int_{\R^n}
\phi(y)
\cdot
\frac{(x-y)(f(x)-f(y))}{|x-y|^{n+\alpha+1}}
\,dx
\,dy.
\end{align*}
Therefore, by adding up the two expressions above, we get that
\begin{align*}
\int_{\R^n} f\,\div^\alpha\phi\,dx
+
\int_{\R^n}\phi\cdot\nabla^\alpha f\,dy
&=
\mu_{n,\alpha}
\int_{\R^n}
\int_{\R^n}
\frac{(x-y)\cdot(f(x)\phi(x)-f(y)\phi(y))}{|x-y|^{n+\alpha+1}}
\,dx
\,dy
\\
&=
\int_{\R^n}
\div^\alpha(f\phi)\,dx=0
\end{align*}
thanks to \cref{res:nabla_alpha_div_alpha_product}\eqref{item:div_alpha_product}. 
This shows that, for any $f \in B^\alpha_{p,1}(\R^n)$, the weak $\alpha$-fractional gradient of $f$ is indeed $\nabla^\alpha f$ and it belongs to $L^p(\R^n; \R^n)$. 
This implies the validity of the continuous inclusion $B^\alpha_{p,1}(\R^n)\subset S^{\alpha,p}(\R^n)$. 
The proof is thus complete. 
\end{proof}

\subsection{The operators  \texorpdfstring{$\nabla^\alpha_{\NL}$}{nablaˆalpha-NL} and \texorpdfstring{$\div^\alpha_{\NL}$}{divˆalpha-NL} on Besov spaces}

Let $\alpha\in(0,1)$.
We recall that the non-local fractional operators $\nabla^\alpha_{\rm NL}$ and $\div^\alpha_{\rm NL}$ are defined by
\begin{equation}
\label{eq:def_nabla_alpha_NL}
\nabla^\alpha_{\NL}(f,g)(x)
=
\mu_{n,\alpha}
\int_{\R^n}\frac{(y-x)(f(y)-f(x))(g(y)-g(x))}{|y-x|^{n+\alpha+1}}\,dy,
\quad
x\in\R^n,
\end{equation}
and
\begin{equation}
\label{eq:def_div_alpha_NL}
\div^\alpha_{\NL}(f,\phi)(x)
=
\mu_{n,\alpha}
\int_{\R^n}\frac{(y-x)\cdot(f(y)-f(x))(\phi(y)-\phi(x))}{|y-x|^{n+\alpha+1}}\,dy,
\quad
x\in\R^n,
\end{equation}
for all $f,g\in\Lip_c(\R^n)$ and $\phi\in\Lip_c(\R^n;\R^n)$, see~\cite{CS19}*{Lemmas~2.6 and~2.7}. 
From \cref{res:D_alpha_D_alpha_NL_int}\eqref{item:D_alpha_NL_int}, we immediately deduce the following integrability result for the operators $\nabla^\alpha_{\rm NL}$ and $\div^\alpha_{\rm NL}$ on Besov spaces.

\begin{corollary}[Integrability of $\nabla^\alpha_{\rm NL}$ and $\div^\alpha_{\rm NL}$]
\label{res:nabla_alpha_NL_div_alpha_NL_int}
Let $\alpha,\beta,\gamma\in(0,1)$ be such that $\alpha=\beta+\gamma$.
Let $p,q,r,s,t\in[1,+\infty]$ be such that $\frac1p+\frac1q=\frac1r$ and $\frac1s+\frac1t=1$. 
The bilinear operators in~\eqref{eq:def_nabla_alpha_NL} and~\eqref{eq:def_div_alpha_NL} can be extended to two bilinear operators (for which we retain the same notation) satisfying 
\begin{equation*}
\begin{split}
\|\nabla^\alpha_{\NL}(f,g)\|_{L^r(\R^n;\,\R^n)}
\le
\begin{cases}
\mu_{n,\alpha}
\,
[f]_{B^{\beta}_{p,s}(\R^n)}
\,
[g]_{B^{\gamma}_{q,t}(\R^n)}
&
\text{
for $f\in b^{\beta}_{p,s}(\R^n)$, $g\in b^{\gamma}_{q,t}(\R^n)$,
}
\\[2mm]
2\mu_{n,\alpha}
\,
\|f\|_{L^p(\R^n)}
\,
[g]_{B^\alpha_{q,1}(\R^n)}	
&
\text{
for $f\in L^p(\R^n)$, $g\in b^{\alpha}_{q,1}(\R^n)$,
}
\\[2mm]
2\mu_{n,\alpha}
\,
[f]_{B^\alpha_{p,1}(\R^n)}	
\,
\|g\|_{L^q(\R^n)}
&
\text{
for
$f\in b^{\alpha}_{p,1}(\R^n)$, $g\in L^q(\R^n)$
}
\end{cases}
\end{split}
\end{equation*}
and, analogously,
\begin{equation*}
\begin{split}
\|\div^\alpha_{\NL}(f,\phi)\|_{L^r(\R^n)}
\le
\begin{cases}
\mu_{n,\alpha}
\,
[f]_{B^{\beta}_{p,s}(\R^n)}
\,
[\phi]_{B^{\gamma}_{q,t}(\R^n;\,\R^n)}
&
\text{
for $f\in b^{\beta}_{p,s}(\R^n)$, $\phi\in b^{\gamma}_{q,t}(\R^n;\R^n)$,
}
\\[2mm]
2\mu_{n,\alpha}
\,
\|f\|_{L^p(\R^n)}
\,
[\phi]_{B^\alpha_{q,1}(\R^n;\,\R^n)}	
&
\text{
for $f\in L^p(\R^n)$, $\phi\in b^{\alpha}_{q,1}(\R^n;\R^n)$,
}
\\[2mm]
2\mu_{n,\alpha}
\,
[f]_{B^\alpha_{p,1}(\R^n)}	
\,
\|\phi\|_{L^q(\R^n;\,\R^n)}
&
\text{
for $f\in b^{\alpha}_{p,1}(\R^n)$, $\phi\in L^q(\R^n;\R^n)$.
}
\end{cases}
\end{split}
\end{equation*}
\end{corollary}

\begin{remark}[Particular cases of \cref{res:nabla_alpha_NL_div_alpha_NL_int}]
As far as we know, two particular cases of \cref{res:nabla_alpha_NL_div_alpha_NL_int} are already known in the literature when $\beta=\alpha/s$ and $\gamma=\alpha/t$ (although for more regular functions). 
The case $r=1$, $s=p$ and $t=q$ is considered in~\cite{CS19}*{Lemmas~2.6 and~2.7} and in the more general~\cite{CS19-2}*{Lemmas~2.4 and~2.5}. 
The case $r=p$, $q=+\infty$, $s=+\infty$ and $t=1$ can be recovered from the computations made in~\cite{KS21}*{Section~2.2, Equation~(2.11)}.	
\end{remark}

Somewhat in analogy with the duality relation~\eqref{eq:ibp_Besov} proved in \cref{res:nabla_div_ibp_Besov}, we can establish the following integration-by-parts formula for the operators $\nabla^\alpha_{\NL}$ and $\div^\alpha_{\NL}$ on Besov spaces.

\begin{lemma}[Duality for $\nabla^\alpha_{\NL}$ and $\div^\alpha_{\NL}$]
\label{res:div_nabla_NL_ibp}
Let $\alpha\in(0,1)$ and let $p,q,r\in[1,+\infty]$ be such that $\frac1p+\frac1q+\frac1r=1$. 
If $f\in L^p(\R^n)$, $g\in b^\alpha_{q,1}(\R^n)$ and $\phi\in L^r(\R^n; \R^n)$, then 
\begin{equation} 
\label{eq:div_nabla_NL_ibp}
\int_{\R^n} f\,\div^{\alpha}_{\rm NL}(g,\phi) \, dx 
=
\int_{\R^n} \phi \cdot \nabla^{\alpha}_{\rm NL}(f, g) \, dx.
\end{equation}
\end{lemma}

\begin{proof}
We start by noticing that both sides of~\eqref{eq:div_nabla_NL_ibp} are well defined by \cref{res:nabla_alpha_NL_div_alpha_NL_int} and H\"older's inequality.
Thanks to \cref{res:D_alpha_D_alpha_NL_int} and Fubini's Theorem, we can write 
\begin{align*}
\int_{\R^n} f\,\div^{\alpha}_{\rm NL}(g,\phi) \, dx 
&=
\mu_{n,\alpha}
\int_{\R^n} 
\int_{\R^n}
f(x)\,
\frac{(y-x)\cdot(g(y)-g(x))(\phi(y)-\phi(x))}{|y-x|^{n+\alpha+1}}\,dy\,dx 
\\
&=
\mu_{n,\alpha}
\int_{\R^n} 
\int_{\R^n}
f(y)\,
\frac{(x-y)\cdot(g(x)-g(y))(\phi(x)-\phi(y))}{|x-y|^{n+\alpha+1}}\,dx\,dy 
\end{align*}
and, similarly,
\begin{align*}
\int_{\R^n} \phi \cdot \nabla^{\alpha}_{\rm NL}(f, g) \, dx
&=
\mu_{n,\alpha}
\int_{\R^n} 
\int_{\R^n}
\phi(x)
\cdot
\frac{(y-x)(f(y)-f(x))(g(y)-g(x))}{|y-x|^{n+\alpha+1}}\,dy\,dx 
\\
&=
\mu_{n,\alpha}
\int_{\R^n} 
\int_{\R^n}
\phi(y)
\cdot
\frac{(x-y)(f(x)-f(y))(g(x)-g(y))}{|x-y|^{n+\alpha+1}}\,dx\,dy.
\end{align*}
Therefore, by exchanging the order of integration and adding up, we get
\begin{align*}
2\int_{\R^n} &f\,\div^{\alpha}_{\rm NL}(g,\phi) \, dx
\\
&=
\mu_{n,\alpha}
\int_{\R^n} 
\int_{\R^n}
\frac{(y-x)\cdot(f(x)-f(y))(g(y)-g(x))(\phi(y)-\phi(x))}{|y-x|^{n+\alpha+1}}\,dy\,dx
\\
&=
\mu_{n,\alpha}
\int_{\R^n} 
\int_{\R^n}
\frac{(y-x)\cdot(f(y)-f(x))(g(y)-g(x))(\phi(x)-\phi(y))}{|y-x|^{n+\alpha+1}}\,dy\,dx
\\
&=
2\int_{\R^n} \phi \cdot \nabla^{\alpha}_{\rm NL}(f, g) \, dx
\end{align*}
and the conclusion follows.
\end{proof}

We end this section with the following result showing that, under suitable Besov regularity assumptions on the function~$h$ and the vector field~$\phi$,
the maps
\begin{equation*}
(f,g)
\mapsto
\int_{\R^n} f\,\nabla^\alpha_{\NL}(g,h)\,dx,
\quad
(f,g)
\mapsto
\int_{\R^n} f\,\div^\alpha_{\NL}(g,\phi)\,dx,
\end{equation*}  
are symmetric. 
Up to our knowledge, this property was first noticed in~\cite{KS21}*{Equation~(2.15)} but only for sufficiently regular functions.

\begin{lemma}[Swapping property of $\nabla^\alpha_{\NL}$ and $\div^\alpha_{\NL}$]
\label{res:nabla_div_NL_swap}
Let $\alpha\in(0,1)$ and let $p,q,r\in[1,+\infty]$ be such that $\frac1p+\frac1q+\frac1r=1$.
If $f\in L^p(\R^n)$, $g\in L^q(\R^n)$, $h\in b^\alpha_{r,1}(\R^n)$ and $\phi\in b^\alpha_{r,1}(\R^n;\R^n)$, then
\begin{equation}
\label{eq:nabla_NL_swap}
\int_{\R^n}f\,\nabla^\alpha_{\NL}(g,h)\,dx
=
\int_{\R^n}
g\,\nabla^\alpha_{\NL}(f,h)\,dx
\end{equation}
and
\begin{equation}
\label{eq:div_NL_swap}
\int_{\R^n}f\,\div^\alpha_{\NL}(g,\phi)\,dx
=
\int_{\R^n}g\,\div^\alpha_{\NL}(f,\phi)\,dx.
\end{equation}
\end{lemma}

\begin{proof}
Thanks to \cref{res:D_alpha_D_alpha_NL_int}\eqref{item:D_alpha_int}, both sides of~\eqref{eq:nabla_NL_swap} are well defined. 
Moreover, also thanks to Fubini's Theorem, we can write
\begin{align*}
\int_{\R^n}f\,\nabla^\alpha_{\NL}(g,h)\,dx
&= \mu_{n,\alpha}
\int_{\R^n}
\int_{\R^n}
f(x)
\,
\frac{(y-x)(g(y)-g(x))(h(y)-h(x))}{|y-x|^{n+\alpha+1}}
\,dy\,dx
\\
&= 
-\mu_{n,\alpha}
\int_{\R^n}
\int_{\R^n}
f(y)
\,
\frac{(y-x)(g(y)-g(x))(h(y)-h(x))}{|y-x|^{n+\alpha+1}}
\,dy\,dx
\end{align*}
by simply swapping the names of the variables.
An analogous reasoning proves that
\begin{align*}
\int_{\R^n}
g\,\nabla^\alpha_{\NL}(f,h)
\,dx
&= \mu_{n,\alpha}
\int_{\R^n}
\int_{\R^n}
g(x)
\,
\frac{(y-x)(f(y)-f(x))(h(y)-h(x))}{|y-x|^{n+\alpha+1}}
\,dy\,dx
\\
&=
- \mu_{n,\alpha}
\int_{\R^n}
\int_{\R^n}
g(y)
\,
\frac{(y-x)(f(y)-f(x))(h(y)-h(x))}{|y-x|^{n+\alpha+1}}
\,dy\,dx.
\end{align*}
We can thus write
\begin{align*}
2\int_{\R^n}&f\,\nabla^\alpha_{\NL}(g,h)\,dx
\\
&=
- \mu_{n,\alpha}
\int_{\R^n}
\int_{\R^n}
\frac{(y-x)(f(y)-f(x))(g(y)-g(x))(h(y)-h(x))}{|y-x|^{n+\alpha+1}}
\,dx\,dy
\\
&=
2
\int_{\R^n}
g\,\nabla^\alpha_{\NL}(f,h)
\,dx
\end{align*} 
and~\eqref{eq:nabla_NL_swap} follows. 
The proof of~\eqref{eq:div_NL_swap} is similar and is thus left to the reader.  
\end{proof}

\section{Leibniz rules involving \texorpdfstring{$BV^{\alpha,p}$}{BVˆ(alpha,p)} functions}

\label{sec:leibniz_BV_alpha_p}

In this section, we detail the proof of our first main result \cref{resi:leibniz_BV_alpha_p}.

\subsection{Products of \texorpdfstring{$BV^{\alpha,p}$}{BVˆ(alpha,p)} and continuous Besov functions}

We begin with the following Leibniz rule for the multiplication of a $BV^{\alpha,p}$ function with a bounded continuous Besov function.

\begin{theorem}[$BV^{\alpha,p}\cdot C_b\cap B^\alpha_{q,1}$ with $\frac1p+\frac1q=1$]
\label{res:BV_alpha_p_leibniz_C_b}
Let $\alpha\in(0,1)$ and let $p,q\in[1,+\infty]$ be such that $\frac1p+\frac1q=1$.
If $f\in BV^{\alpha, p}(\R^n)$ and $g\in C_b(\R^n)\cap B^\alpha_{q,1}(\R^n)$, then
$fg\in BV^{\alpha,r}(\R^n)$ for all $r\in[1,p]$, with
\begin{equation} 
\label{eq:BV_alpha_p_leibniz_C_b}
D^{\alpha}(fg) 
= 
g\, D^{\alpha} f 
+ 
f\,
\nabla^{\alpha}g\,
\Leb{n} 
+ 
\nabla^{\alpha}_{\rm NL}(f,g)\,\Leb{n}
\quad 
\text{in}\ \M (\R^n; \R^{n}).
\end{equation}
In addition,
\begin{equation}
\label{eq:BV_alpha_p_leibniz_C_b_D_alpha_zero}
D^\alpha(fg)(\R^n)=0,
\qquad
\int_{\R^n}\nabla^\alpha_{\rm NL}(f,g)\,dx=0,
\end{equation}
and 
\begin{equation}
\label{eq:BV_alpha_p_leibniz_C_b_ibp}
\int_{\R^n}
f\,\nabla^\alpha g\,dx
=
-
\int_{\R^n}
g\,d D^\alpha f.
\end{equation}
\end{theorem}

In the proof of \cref{res:BV_alpha_p_leibniz_C_b}, we take advantage of the following Leibniz rules for the operators $\nabla^\alpha$ and $\div^\alpha$ in Besov spaces.
\cref{res:nabla_alpha_div_alpha_leibniz_besov} below generalizes~\cite{CS19}*{Lemmas~2.5 and~2.6} and~\cite{CS19-2}*{Lemmas~2.4 and~2.5}.

\begin{lemma}[Leibniz rules for $\nabla^\alpha$ and $\div^\alpha$]
\label{res:nabla_alpha_div_alpha_leibniz_besov}
Let $\alpha\in(0,1)$ and let $p,q,r\in[1,+\infty]$ be such that $\frac1p+\frac1q=\frac1r$.

\begin{enumerate}[(i)]

\item
\label{item:nabla_alpha_leibniz_besov} 
If $f\in B^\alpha_{p,1}(\R^n)$ and $g\in B^\alpha_{q,1}(\R^n)$, then
\begin{equation}
\label{eq:nabla_alpha_leibniz_besov}
\nabla^\alpha(fg)
=
g\,\nabla^\alpha f
+
f\,\nabla^\alpha g
+
\nabla^\alpha_{\rm NL}(f,g)
\quad
\text{in}\ L^r(\R^n;\R^n).
\end{equation}

\item
\label{item:div_alpha_leibniz_besov} 
If $f\in B^\alpha_{p,1}(\R^n)$ and $\phi\in B^\alpha_{q,1}(\R^n;\R^n)$, then
\begin{equation}
\label{eq:div_alpha_leibniz_besov}
\div^\alpha(f\phi)
=
\phi\cdot\nabla^\alpha f
+
f\,\div^\alpha\phi
+
\div^\alpha_{\rm NL}(f,\phi)
\quad
\text{in}\ L^r(\R^n).
\end{equation}	

\end{enumerate}
\end{lemma}

\begin{proof}
By \cref{res:nabla_alpha_div_alpha_in_L_p}, \cref{res:nabla_alpha_div_alpha_product}\eqref{item:nabla_alpha_product} and \cref{res:nabla_alpha_NL_div_alpha_NL_int}, we already know that each of the terms appearing in~\eqref{eq:nabla_alpha_leibniz_besov} and~\eqref{eq:div_alpha_leibniz_besov} is a well defined $L^r$-function.
Thanks to \cref{res:D_alpha_D_alpha_NL_int} and Fubini's Theorem, we can thus write
\begin{align*}
\nabla^\alpha(fg)&(x)
=
\mu_{n,\alpha}
\int_{\R^n}
\frac{(y-x)(f(y)g(y)-f(x)g(x))}{|y-x|^{n+\alpha+1}}
\,dy
\\
&=
\mu_{n,\alpha}
\int_{\R^n}
\frac{(y-x)(f(y)g(y)-f(x)g(y)+f(x)g(y)-f(x)g(x))}{|y-x|^{n+\alpha+1}}
\,dy
\\
&=
\mu_{n,\alpha}
\int_{\R^n}
g(y)\,
\frac{(y-x)(f(y)-f(x))}{|y-x|^{n+\alpha+1}}
\,dy
+
f(x)\,\nabla^\alpha g(x)
\\
&=
\mu_{n,\alpha}
\int_{\R^n}
\frac{(y-x)(f(y)-f(x))(g(y)-g(x))}{|y-x|^{n+\alpha+1}}
\,dy
+
g(x)\,\nabla^\alpha f(x)
+
f(x)\,\nabla^\alpha g(x)
\\
&=
g(x)\,\nabla^\alpha f(x)
+
f(x)\,\nabla^\alpha g(x)
+
\nabla^\alpha_{\rm NL}(f,g)(x)
\end{align*}
for a.e.\ $x\in\R^n$, proving~\eqref{eq:nabla_alpha_leibniz_besov}.
The proof of~\eqref{eq:div_alpha_leibniz_besov} is similar and is left to the reader.
\end{proof}

\begin{proof}[Proof of \cref{res:BV_alpha_p_leibniz_C_b}]
Since $g\in L^q(\R^n)\cap L^\infty(\R^n)$, we clearly have that $fg\in L^1(\R^n)\cap L^p(\R^n)$ by H\"older's inequality.
We divide the proof in two steps.

\smallskip

\textit{Step~1: proof of~\eqref{eq:BV_alpha_p_leibniz_C_b}}.
Let $\phi\in\Lip_c(\R^n;\R^n)$ be given.
By \cref{res:nabla_alpha_div_alpha_leibniz_besov}\eqref{item:div_alpha_leibniz_besov}, we can write 
\begin{equation*}
\div^\alpha(g\phi)
=
g\,\div^\alpha\phi
+
\phi\cdot\nabla^\alpha g
+\div_{\rm NL}^\alpha(g,\phi)
\quad
\text{in}\
L^q(\R^n),
\end{equation*} 
so that
\begin{align*}
\int_{\R^n}fg\,\div^\alpha\phi\,dx 
=
\int_{\R^n}f\,
\div^\alpha(g\phi)
\,dx
-
\int_{\R^n}
f\phi\cdot\nabla^\alpha g
\,dx
-
\int_{\R^n}
f\,\div^\alpha_{\rm NL}(g,\phi)\,
dx.
\end{align*}
By \cref{res:div_nabla_NL_ibp}, we have that
\begin{equation*}
\int_{\R^n}
f\,\div^\alpha_{\rm NL}(g,\phi)\,
dx
=
\int_{\R^n}
\phi\cdot\nabla^\alpha_{\rm NL}(f,g)\,
dx.
\end{equation*}
Now let $(f_\eps)_{\eps>0}\subset BV^{\alpha,p}(\R^n)\cap C^\infty(\R^n)$ be given by $f_\eps=\rho_\eps*f$ for all $\eps>0$. 
In particular, we have $f_{\eps} \in W^{1, p}(\R^n)$ for each $\eps > 0$, and we notice that $W^{1,p}(\R^n) \subset B^{\alpha}_{p,q}(\R^n)$ for all $\alpha \in (0,1)$ and $p, q \in [1, +\infty]$, see~\cite{Leoni17}*{Theorem~17.33}. 
As a consequence, we have $f_\eps\in B^\alpha_{p,1}(\R^n)$ for each $\eps>0$.
Since $g\phi\in B^\alpha_{q,1}(\R^n;\R^n)$ for each $\eps>0$, by \cref{res:nabla_div_ibp_Besov} we can write
\begin{equation*}
\int_{\R^n}f_\eps\,
\div^\alpha(g\phi)
\,dx
=
-
\int_{\R^n}g\phi\cdot
\nabla^\alpha f_\eps\,dx
\end{equation*}
for all $\eps>0$.
On the one side, we have
\begin{equation*}
\lim_{\eps\to0^+}
\int_{\R^n}f_\eps\,
\div^\alpha(g\phi)
\,dx
=
\int_{\R^n}
f\,
\div^\alpha(g\phi)
\,dx
\end{equation*}
by H\"older's inequality in the case $p<+\infty$ and by the Dominated Convergence Theorem in the case $p=+\infty$.
On the other side, since $g\phi\in C_c(\R^n;\R^n)$, we also have
\begin{equation*}
\lim_{\eps\to0^+}
\int_{\R^n}g\phi\cdot
\nabla^\alpha f_\eps\,dx
=
\int_{\R^n}g\phi\cdot
dD^\alpha f
\end{equation*}
since $D^\alpha f_\eps=\rho_\eps*D^\alpha f\weakto D^\alpha f$ in $\M (\R^n;\R^n)$ as $\eps\to0^+$, thanks to \cite{CSS21}*{Theorem 3.2}.
We thus conclude that 
\begin{equation*}
\int_{\R^n}f\,
\div^\alpha(g\phi)
\,dx
=
-
\int_{\R^n}g\phi\cdot
dD^\alpha f,
\end{equation*}
so that
\begin{equation*}
\int_{\R^n}fg\,\div^\alpha\phi\,dx 
=
-
\int_{\R^n}g\phi\cdot
dD^\alpha f
-
\int_{\R^n}
f\phi\cdot\nabla^\alpha g
\,dx
-
\int_{\R^n}
\phi\cdot\nabla^\alpha_{\rm NL}(f,g)\,
dx
\end{equation*}
for any $\phi\in\Lip_c(\R^n;\R^n)$ and~\eqref{eq:BV_alpha_p_leibniz_C_b} immediately follows by a standard approximation argument which allows us to pass to test functions in $C_c(\R^n; \R^n)$. 

\smallskip

\textit{Step~2: proof of~\eqref{eq:BV_alpha_p_leibniz_C_b_D_alpha_zero} and~\eqref{eq:BV_alpha_p_leibniz_C_b_ibp}}.
Since $fg\in BV^{\alpha}(\R^n)$ by Step~1, the first equation in~\eqref{eq:BV_alpha_p_leibniz_C_b_D_alpha_zero} readily follows from \cref{res:zero_total_mesure}.
Moreover, since obviously $\div^\alpha_{\rm NL}(g,v)=0$ for all $v\in\R^n$,
by \cref{res:div_nabla_NL_ibp} we get
\begin{align*}
v\cdot
\int_{\R^n}\nabla^\alpha_{\rm NL}(f,g)\,dx
=
\int_{\R^n}v\cdot\nabla^\alpha_{\rm NL}(f,g)\,dx
=
\int_{\R^n}f\,\div^\alpha_{\rm NL}(g,v)\,dx=0
\end{align*} 
for all $v\in\R^n$ and also the second equation in~\eqref{eq:BV_alpha_p_leibniz_C_b_D_alpha_zero} immediately follows.
By combining~\eqref{eq:BV_alpha_p_leibniz_C_b} with~\eqref{eq:BV_alpha_p_leibniz_C_b_D_alpha_zero}, we get~\eqref{eq:BV_alpha_p_leibniz_C_b_ibp} and the proof is complete.  
\end{proof}

%
%
%

\subsection{Products of \texorpdfstring{$BV^{\alpha,p}$}{BVˆ(alpha,p)} and Besov functions}

In the remaining part of the present section we aim to relax the assumptions on the Besov function~$g$ in \cref{res:BV_alpha_p_leibniz_C_b}, proving our first main result \cref{resi:leibniz_BV_alpha_p}. 

Let us observe that the continuity of the Besov function~$g$ in \cref{res:BV_alpha_p_leibniz_C_b} can be trivially dropped for $p\in\left[1,\frac n{n-\alpha}\right)$, since for the conjugate exponent $q\in\left(\frac n\alpha,+\infty\right]$ we know that $B^\alpha_{q,1}(\R^n)\subset C_b(\R^n)$ by the Sobolev Embedding Theorem, see~\cite{AF03}*{Theorem 7.34(c)} and \cite{Leoni17}*{Theorem~17.52}.

\begin{corollary}[$BV^\alpha\cdot B^\alpha_{q,1}$ for $1\le p<\frac n{n-\alpha}$]
\label{res:giovannino}
Let $\alpha\in(0,1)$, $p\in\left[1,\frac n{n-\alpha}\right)$ and set $q=\frac p{p-1}\in\left(\frac n\alpha,+\infty\right]$.
If $f\in BV^{\alpha, p}(\R^n)$ and $g\in  B^\alpha_{q,1}(\R^n)$, then
$fg\in BV^{\alpha,r}(\R^n)$, with $r\in[1,p]$, satisfies~\eqref{eq:BV_alpha_p_leibniz_C_b}.
In addition, equalities~\eqref{eq:BV_alpha_p_leibniz_C_b_D_alpha_zero}
and~\eqref{eq:BV_alpha_p_leibniz_C_b_ibp} hold.
\end{corollary}

In order to keep removing the continuity assumption on the Besov function $g$ for higher exponents $p\in\left[\frac n{n-\alpha},+\infty\right]$, we have to rely on the absolute continuity properties of the fractional variation established in~\cite{CSS21}*{Theorem~1.1}. 
To this purpose, we recall the notion of $(\alpha,p)$-capacity, for whose properties we refer to~\cite{CSS21}*{Section~5}.

\begin{definition}[The $(\alpha,p)$-capacity]
Let $\alpha\in(0,1)$ and $p\in[1,+\infty)$.
We let 
\begin{equation*}
\Capa_{\alpha,p}(K)
=
\inf
\set*{\|f\|_{S^{\alpha,p}(\R^n)}^p : f\in C^\infty_c(\R^n),\ f\ge\chi_K}
\end{equation*}
be the \emph{$(\alpha,p)$-capacity} of the compact set $K\subset\R^n$.
\end{definition}

We begin with the following result, in which the continuity assumption on the Besov function~$g$ in \cref{res:BV_alpha_p_leibniz_C_b} is dropped for exponents $p\in\left[\frac n{n-\alpha},\frac1{1-\alpha}\right)$, provided that in~\eqref{eq:BV_alpha_p_leibniz_C_b} and~\eqref{eq:BV_alpha_p_leibniz_C_b_ibp} the function~$g$ is replaced with its precise representative~$g^\star $.
We emphasize that, since $g$ is a Besov function, its precise representative satisfies
\begin{equation*}
\lim_{r \to 0^+} \aint_{B_r(x)} |g(y) - g^{*}(x)| \, dy = 0
\end{equation*} 
for every point $x\in\R^n$ outside a set which is negligible with respect to a suitable $(\beta,q)$-capacity, see~\cite{CSS21}*{Theorem~5.4} for a more detailed explanation.

\begin{corollary}[$BV^{\alpha,p}\cdot L^\infty\cap B^\alpha_{q,1}$ for $p\in[\frac n{n-\alpha},\frac1{1-\alpha})$]
\label{res:BV_alpha_p_leibniz_Sob_embedding}
Let $\alpha\in(0,1)$, $p\in\left[\frac n{n-\alpha},\frac1{1-\alpha}\right)$ and set $q=\frac p{p-1}\in\left(\frac1\alpha,\frac n\alpha\right]$.
If $f\in BV^{\alpha, p}(\R^n)$ and $g\in L^\infty(\R^n)\cap B^\alpha_{q,1}(\R^n)$, then
$fg\in BV^{\alpha,r}(\R^n)$ for all $r \in [1, p]$, with
\begin{equation}
\label{eq:pallina} 
D^{\alpha}(fg) 
= 
g^\star  D^{\alpha} f 
+ 
f\,
\nabla^{\alpha}g\,
\Leb{n} 
+ 
\nabla^{\alpha}_{\rm NL}(f,g)\,\Leb{n}
\quad 
\text{in}\ \M (\R^n; \R^{n}).
\end{equation}
In addition,
\begin{equation}
\label{eq:ping}
D^\alpha(fg)(\R^n)=0,
\qquad
\int_{\R^n}\nabla^\alpha_{\rm NL}(f,g)\,dx=0,
\end{equation}
and 
\begin{equation}
\label{eq:pong}
\int_{\R^n}
f\,\nabla^\alpha g\,dx
=
-
\int_{\R^n}
g^\star \,d D^\alpha f.
\end{equation}
\end{corollary}

\begin{proof}
Note that $n\ge2$ trivially. 
Let $(\rho_\eps)_{\eps>0}$ be a family of standard mollifiers (see~\cite{CS19}*{Section~3.3} for example) and let $g_\eps=\rho_\eps*g$ for all $\eps>0$. 
Since $g_\eps\in  C_b(\R^n)\cap B^\alpha_{q,1}(\R^n)$ for all $\eps>0$ thanks to \cite{Leoni17}*{Proposition~17.12}, by \cref{res:BV_alpha_p_leibniz_C_b} we immediately get that $fg_\eps\in BV^{\alpha,p}(\R^n)$ with
\begin{equation*}
D^{\alpha}(fg_\eps) 
= 
g_\eps\, D^{\alpha} f 
+ 
f\,
\nabla^{\alpha}g_\eps\,
\Leb{n} 
+ 
\nabla^{\alpha}_{\rm NL}(f,g_\eps)\,\Leb{n}
\quad 
\text{in}\ \M (\R^n; \R^{n}).
\end{equation*}
By \cref{res:D_alpha_D_alpha_NL_int} and \cref{res:nabla_alpha_div_alpha_in_L_p}, we have $\nabla^{\alpha}g \in L^{q}(\R^n; \R^n)$ and it is not difficult to recognize that 
\begin{equation*}
\nabla^\alpha g_\eps
=
\rho_\eps*\nabla^\alpha g
\quad
\text{in}\ 
L^q(\R^n;\R^n)
\end{equation*}
and 
\begin{align*}
\|\nabla^{\alpha}_{\rm NL}(f,g_\eps)
-
\nabla^{\alpha}_{\rm NL}(f,g)\|_{L^1(\R^n;\,\R^n)}
&=
\|\nabla^{\alpha}_{\rm NL}(f,g_\eps-g)
\|_{L^1(\R^n;\,\R^n)}
\\
&\le
2\mu_{n,\alpha}
\,
\|f\|_{L^p(\R^n)}
\,
[g-g_\eps]_{B^\alpha_{q,1}(\R^n)}
\end{align*} 
for all $\eps>0$. 
Since $\rho_\eps*\nabla^\alpha g \to \nabla^{\alpha} g$ in $L^{q}(\R^n; \R^n)$ and, by~\cite{Leoni17}*{Proposition~17.12}, $[g-g_\eps]_{B^\alpha_{q,1}(\R^n)} \to 0$, the only thing we have to prove is that 
\begin{equation}
\label{eq:leibniz_claim_g_eps_limit}
\lim_{\eps\to0^+}
g_\eps(x)
=
g^\star (x)
\quad
\text{for $|D^\alpha f|$-a.e.}\ x\in\R^n.
\end{equation} 
Now, since $g \in S^{\alpha, q}(\R^n)$ and
$1<q\le\frac n\alpha$, by~\cite{CSS21}*{Theorem~5.4} we get that
\begin{equation}
\label{eq:g_eps_limit_quasievery_point}
\lim_{\eps\to0^+}
g_\eps(x)
=
g^\star (x)
\quad
\text{for all}\ x\in\R^n\setminus D_g,
\end{equation}
for some set $D_g\subset\R^n$ with $\Capa_{\alpha,q}(D_g)=0$.
Recall that $\Haus{n-\alpha q+\delta}\ll\Capa_{\alpha,q}$ for any $\delta>0$ sufficiently small by~\cite{AH96}*{Theorem~5.1.13 and Corollary~5.1.14}.
Since  
$p<\frac1{1-\alpha}$,
we have
$|D^\alpha f|\ll\Haus{n-1}$
by~\cite{CSS21}*{Theorem~1.1(i)}.
Observing that $n-1>n-\alpha q$, we conclude that $|D^\alpha f|(D_g)=0$, thus proving~\eqref{eq:leibniz_claim_g_eps_limit}. 
Finally, \eqref{eq:ping} and \eqref{eq:pong} can be proved as \eqref{eq:BV_alpha_p_leibniz_C_b_D_alpha_zero} and \eqref{eq:BV_alpha_p_leibniz_C_b_ibp} in \cref{res:BV_alpha_p_leibniz_C_b}.
\end{proof} 

We continue to remove the continuity assumption on the Besov function~$g$ in \cref{res:BV_alpha_p_leibniz_C_b} for larger exponents $p\in\left[\frac1{1-\alpha},+\infty\right]$.
We treat the cases $p<+\infty$ and $p=+\infty$ separately, see \cref{res:remove_C_p_finite} and \cref{res:remove_C_p_infty} below. 
Since we still want to keep exploiting the absolute continuity properties of the fractional variation proved in~\cite{CSS21}*{Theorem~1.1}, we are forced to assume some additional fractional smoothness on the function~$g$.

\begin{corollary}[$BV^{\alpha,p}\cdot L^\infty\cap B^\beta_{q,1}$ for $p\in[\frac1{1-\alpha},+\infty)$]
\label{res:remove_C_p_finite} 
Let $\alpha\in(0,1)$, $p\in\left[\frac1{1-\alpha},+\infty\right)$ and set $q=\frac p{p-1}\in\left(1,\frac1\alpha\right]$. 
Let $\beta\in\left(\beta_{n,p,\alpha},1\right)$, where
\begin{equation*}
\beta_{n,p,\alpha} 
= 
\begin{cases} 1 - \frac{1}{p} 
& 
\text{ if } n \ge 2\
\text{and}\
p\in\left[\frac 1{1-\alpha},\frac n{1-\alpha}\right), 
\\[4mm]
\left(1 - \frac{1}{p} \right)
\left(\alpha+\frac np\right)
& 
\text{ if } p\in\left[\frac n{1-\alpha},+\infty\right).
\end{cases}
\end{equation*}
If $f\in BV^{\alpha,p}(\R^n)$ and $g\in L^{\infty}(\R^n)\cap B^\beta_{q,1}(\R^n)$, then $fg\in BV^{\alpha,r}(\R^n)$ for all $r \in [1,p]$ and satisfies~\eqref{eq:pallina}.
In addition, equalities~\eqref{eq:ping} and~\eqref{eq:pong} hold.
\end{corollary}

\begin{proof}
Let $(\rho_\eps)_{\eps>0}$ be a family of standard mollifiers (see~\cite{CS19}*{Section~3.3} for example) and let $g_\eps=\rho_\eps*g$ for all $\eps>0$. 
The proof now goes as for \cref{res:BV_alpha_p_leibniz_Sob_embedding}, again the only thing we need to prove being~\eqref{eq:leibniz_claim_g_eps_limit}.
We now distinguish two cases.

\smallskip

\textit{Case~1: $n \ge 2$ and $p\in\left[\frac 1{1-\alpha},\frac n{1-\alpha}\right)$}.
By~\cite{CSS21}*{Theorem~5.4}, we know that the limit in~\eqref{eq:g_eps_limit_quasievery_point} is valid for some set $D_g\subset\R^n$ such that $\Capa_{\beta,q}(D_g)=0$.
Recall that $\Haus{n-\beta q+\delta}\ll\Capa_{\beta,q}$ for any $\delta>0$ by~\cite{AH96}*{Theorem~5.1.13 and Corollary~5.1.14}.
Since  
$p<\frac n{1-\alpha}$,
we have
$|D^\alpha f|\ll\Haus{n-1}$
by~\cite{CSS21}*{Theorem~1.1(i)}. 
Observing that $n-1>n-\beta q$, since $\beta > \frac{1}{q}=1-\frac1p$, we get that $|D^\alpha f|(D_g)=0$ and the conclusion follows.

\smallskip

\textit{Case~2: $p\in\left[\frac n{1-\alpha},+\infty\right)$}.
If $n=1$ and $p=\frac1{1-\alpha}$, then $\beta\in(\alpha,1)$ and $q=\frac1\alpha>\frac1\beta$, so that $B^\beta_{q,1}(\R)\subset C_b(\R)$ by the Sobolev Embedding Theorem~\cite{AF03}*{Theorem~7.34(c)}. 
Hence the limit in~\eqref{eq:g_eps_limit_quasievery_point} is valid for all $x\in\R$ and the conclusion follows immediately.
We can thus assume either $n=1$ and $p>\frac1{1-\alpha}$ or $n\ge2$.
By~\cite{CSS21}*{Theorem~5.4}, we know that the limit in~\eqref{eq:g_eps_limit_quasievery_point} is valid for some set $D_g\subset\R^n$ such that $\Capa_{\beta,q}(D_g)=0$.
Recall that $\Haus{n-\beta q+\delta}\ll\Capa_{\beta,q}$ for any $\delta>0$ by~\cite{AH96}*{Theorem~5.1.13 and Corollary~5.1.14}.
Since  
$p\ge\frac n{1-\alpha}$,
we have
$|D^\alpha f|\ll\Haus{\frac nq-\alpha}$
by~\cite{CSS21}*{Theorem~1.1(ii)}. 
Observing that $\frac nq-\alpha>n-\beta q$ if and only if $\beta > \frac{n+\alpha}{q} - \frac{n}{q^2} = \left(1 - \frac{1}{p} \right)
\left(\alpha+\frac np\right)
$, we get that $|D^\alpha f|(D_g)=0$ and the conclusion follows.
\end{proof}

We notice that $\beta_{n,p,\alpha} \to \alpha^+$ for $p \to + \infty$.
Consequently, at least heuristically, if $p = + \infty$, then we can take $g \in L^{\infty}(\R^n) \cap W^{\beta, 1}(\R^n)$ for all $\beta \in [\alpha, 1)$. 
Actually, since $W^{\beta, 1}(\R^n) \subset W^{\alpha, 1}(\R^n)$ for all $\beta \in [\alpha, 1)$, this is equivalent to choosing $g\in L^{\infty}(\R^n)\cap W^{\alpha,1}(\R^n)$. The following result provides a rigorous formulation of this observation.

\begin{corollary}[$BV^{\alpha,\infty}\cdot L^\infty\cap W^{\alpha,1}$]
\label{res:remove_C_p_infty} 
Let $\alpha\in(0,1)$.
If $f\in BV^{\alpha, \infty}(\R^n)$ and $g\in L^{\infty}(\R^n)\cap W^{\alpha,1}(\R^n)$, then $fg\in BV^{\alpha,1}(\R^n)\cap BV^{\alpha,\infty}(\R^n)$ satisfies~\eqref{eq:pallina}.
In addition, equalities~\eqref{eq:ping} and~\eqref{eq:pong} hold.
\end{corollary}

\begin{proof}
Note that $fg\in L^1(\R^n)\cap L^{\infty}(\R^n)$.
Let $(\rho_\eps)_{\eps>0}$ be a family of standard mollifiers (see~\cite{CS19}*{Section~3.3} for example) and let $g_\eps=\rho_\eps*g$ for all $\eps>0$. 
The proof now goes as for \cref{res:BV_alpha_p_leibniz_Sob_embedding}, again the only thing we need to prove being~\eqref{eq:leibniz_claim_g_eps_limit}.
Now, by~\cite{PS20}*{Proposition~3.1}, we know that the limit in~\eqref{eq:g_eps_limit_quasievery_point} is valid for some set $D_g\subset\R^n$ such that $\Haus{n-\alpha}(D_g)=0$.
Since  
$p=+\infty$ and 
$|D^\alpha f|\ll\Haus{n-\alpha}$
by~\cite{CSS21}*{Theorem~1.1(ii)}, the conclusion immediately follows.
\end{proof}

\subsection{Some Gauss--Green formulas}

\label{subsec:GG_formulas}

We conclude this section by briefly discussing some Gauss--Green formulas in the whole space.

We emphasize the following particular cases of \cref{res:remove_C_p_infty}.
Note that \cref{res:ggf_frac}  below generalizes the fractional Gauss--Green formula proved in~\cite{CS19}*{Theorem~4.2} on~$\R^n$. 
For the definitions of the fractional reduced boundary $\redb^\alpha E$ of a set~$E$ with finite fractional Caccioppoli $\alpha$-perimeter and of its measure theoretic inner unit fractional normal $\nu_E^{\alpha}$, we refer the reader to~\cite{CS19}*{Definition 4.7} 

\begin{corollary}[Generalized fractional Gauss--Green formulas]
\label{res:ggf_frac}
Let $\alpha\in(0,1)$.
\begin{enumerate}[(i)]

\item 
If $f\in BV^{\alpha,\infty}(\R^n)$, then
\begin{equation*}
\int_{E^{1}}
dD^\alpha f
=
-\int_{\redb^\alpha E}
f\,\nu^\alpha_E\,|\nabla^\alpha\chi_E|
\,dx
\end{equation*}
for all measurable sets $E\subset\R^n$ such that $\chi_E\in W^{\alpha,1}(\R^n)$, where
\begin{equation*}
E^{1}
=
\set*{x\in\R^n : \exists\lim_{r\to0^+}\frac{|E\cap B_r(x)|}{|B_r(x)|}=1}.
\end{equation*} 	

\item
If $f\in W^{\alpha,1}(\R^n)\cap L^\infty(\R^n)$, then 
\begin{equation*}
\int_E 
\nabla^\alpha f
\,dx
=
-
\int_{\redb^\alpha E}
f^\star\,\nu^\alpha_E
\,
d|D^\alpha\chi_E|
\end{equation*}
for all measurable sets $E\subset\R^n$ such that $|D^\alpha\chi_E|(\R^n)<+\infty$.

\end{enumerate}
\end{corollary}

\begin{remark}
By \cref{res:remove_C_p_infty},  if $f\in W^{\alpha,1}(\R^n)\cap L^\infty(\R^n)$, then
\begin{equation*}
\int_{\R^n} f\, \nabla^{\alpha} f \, dx = 0.
\end{equation*}
Consequently, we have
\begin{equation*}
\int_{E \cap \redb^\alpha E} \nu^\alpha_E\,|\nabla^\alpha\chi_E| \, dx = \int_{(\redb^\alpha E) \setminus E} \nu^\alpha_E\,|\nabla^\alpha\chi_E| \, dx = 0
\end{equation*}
for all measurable sets $E\subset\R^n$ such that $\chi_E\in W^{\alpha,1}(\R^n)$.
\end{remark}

As it is apparent from \cref{res:remove_C_p_infty} and \cref{res:ggf_frac}, because of the problem of the well-posedness of the product $g^\star D^\alpha f$, we are forced to pair a function $f$ with bounded fractional variation together with a function $g$ possessing a stronger $W^{\alpha,1}$ regularity.
\cref{res:gg_S_alpha_1} below shows that, if we are not interested in the fractional variation measure, both functions $f$ and $g$ can have the same $S^{\alpha,1}$ regularity.

\begin{proposition}[Gauss--Green formula in $S^{\alpha,1}\cap L^\infty$]
\label{res:gg_S_alpha_1}
Let $\alpha\in(0,1)$.
If $f,g\in S^{\alpha,1}(\R^n)\cap L^\infty(\R^n)$, then
\begin{equation}
\label{eq:gg_S_alpha_1}
\int_{\R^n}f\,\nabla^\alpha g\,dx
=
-\int_{\R^n}g\,\nabla^\alpha f\,dx.
\end{equation}	
\end{proposition}

\begin{proof}
Since $f,g\in S^{\alpha,1}(\R^n)\cap L^\infty(\R^n)$, both sides of~\eqref{eq:gg_S_alpha_1} are well defined.
Moreover, we can find $(f_k)_{k\in\N},(g_k)_{k\in\N}\subset C^\infty_c(\R^n)$ such that $f_k\to f$ and $g_k\to g$ in $S^{\alpha,1}(\R^n)$ and pointwise a.e.\ in~$\R^n$ as $k\to+\infty$, with $\|f_k\|_{L^\infty(\R^n)}\le\|f\|_{L^\infty(\R^n)}$ and $\|g_k\|_{L^\infty(\R^n)}\le\|g\|_{L^\infty(\R^n)}$ for all $k\in\N$.
This statement easily follows from the proofs of~\cite{CS19}*{Theorems~3.22 and~3.23}, so we leave the details to the interested reader.
We clearly have that
\begin{equation*}
\int_{\R^n}f_k\,\nabla^\alpha g_k\,dx
=
-\int_{\R^n}g_k\,\nabla^\alpha f_k\,dx
\end{equation*}
for all $k\in\N$.
We can now  observe that 
\begin{align*}
\bigg|
\int_{\R^n}f_k\,\nabla^\alpha g_k\,dx
-&
\int_{\R^n}f\,\nabla^\alpha g\,dx
\,
\bigg|
\le
\int_{\R^n}|f_k|\,|\nabla^\alpha g_k- \nabla^\alpha g|\,dx
+
\int_{\R^n}|f_k-f|\,|\nabla^\alpha g|\,dx
\\
&\le
\|f\|_{L^\infty(\R^n)}
\,
\|\nabla^\alpha g_k- \nabla^\alpha g\|_{L^1(\R^n;\,\R^n)}
+
\int_{\R^n}|f_k-f|\,|\nabla^\alpha g|\,dx
\end{align*}
so that 
\begin{equation*}
\lim_{k\to+\infty}
\int_{\R^n}f_k\,\nabla^\alpha g_k\,dx
=
\int_{\R^n}f\,\nabla^\alpha g\,dx
\end{equation*}
by the Dominated Convergence Theorem.
A similar reasoning provides the other limit
\begin{equation*}
\lim_{k\to+\infty}
\int_{\R^n}g_k\,\nabla^\alpha f_k\,dx
=
\int_{\R^n}g\,\nabla^\alpha f\,dx	
\end{equation*}
and the proof is thus complete.
\end{proof}

\section{Leibniz rules for \texorpdfstring{$S^{\alpha,p}$}{Sˆ(alpha,p)} functions}

\label{sec:leibniz_bessel}

In this section, we prove our second and third main results, \cref{resi:lebniz_bessel_conjugate} and \cref{resi:lebniz_bessel_bounded} respectively.

\subsection{Products of \texorpdfstring{$S^{\alpha,p}$}{Sˆ(alpha,p)} and Besov functions}

Similarly to \cref{res:BV_alpha_p_leibniz_C_b}, we can prove the following Leibniz rule for the product of $S^{\alpha,p}$ and Besov functions.
\cref{res:S_alpha_p_leibniz_besov} below generalizes~\cite{BCM20}*{Lemma~3.4} and~\cite{KS21}*{Lemma~2.11}.

\begin{theorem}[$S^{\alpha,p}\cdot B^\alpha_{q,1}$]
\label{res:S_alpha_p_leibniz_besov}
Let $\alpha\in(0,1)$ and let $p,q,r\in[1,+\infty]$ be such that $\frac1p+\frac1q=\frac1r$.
If $f\in S^{\alpha,p}(\R^n)$ and $g\in B^\alpha_{q,1}(\R^n)$, then $fg\in S^{\alpha,r}(\R^n)$ with
\begin{equation}
\nabla^\alpha(fg)
=
g\,\nabla^\alpha f
+
f\,\nabla^\alpha g
+
\nabla^\alpha_{\rm NL}(f,g)\quad
\text{in}\
L^r(\R^n;\R^n).	
\end{equation}	
\end{theorem}

\begin{proof}
We clearly have that $fg\in L^r(\R^n)$ by H\"older's (generalized) inequality.
We now let $\phi\in\Lip_c(\R^n;\R^n)$ be given. 
By \cref{res:nabla_alpha_div_alpha_leibniz_besov}\eqref{item:div_alpha_leibniz_besov}, we can write 
\begin{equation*}
\div^\alpha(g\phi)
=
g\,\div^\alpha\phi
+
\phi\cdot\nabla^\alpha g
+\div_{\rm NL}^\alpha(g,\phi)
\quad
\text{in}\
L^{p'}(\R^n),
\end{equation*} 
where $\frac1p+\frac1{p'}=1$, so that
\begin{align*}
\int_{\R^n}fg\,\div^\alpha\phi\,dx 
=
\int_{\R^n}f\,
\div^\alpha(g\phi)
\,dx
-
\int_{\R^n}
f\phi\cdot\nabla^\alpha g
\,dx
-
\int_{\R^n}
f\,\div^\alpha_{\rm NL}(g,\phi)\,
dx.
\end{align*}
By \cref{res:div_nabla_NL_ibp}, we have that
\begin{equation*}
\int_{\R^n}
f\,\div^\alpha_{\rm NL}(g,\phi)\,
dx
=
\int_{\R^n}
\phi\cdot\nabla^\alpha_{\rm NL}(f,g)\,
dx.
\end{equation*}
Now let $(\rho_\eps)_{\eps>0}$ be a family of standard mollifiers (see~\cite{CS19}*{Section~3.3} for example) and let $(f_\eps)_{\eps>0}\subset S^{\alpha,p}(\R^n)\cap C^\infty(\R^n)$ be given by $f_\eps=f*\rho_\eps$ for all $\eps>0$. 
Arguing as in the proof of \cref{res:BV_alpha_p_leibniz_C_b}, we see that $f_\eps\in B^\alpha_{p,1}(\R^n)$ for each $\eps>0$. It is easy to check that $g\phi\in B^\alpha_{p',1}(\R^n;\R^n)$, so that by \cref{res:nabla_div_ibp_Besov} we can write
\begin{equation*}
\int_{\R^n}f_\eps\,
\div^\alpha(g\phi)
\,dx
=
-
\int_{\R^n}g\phi\cdot
\nabla^\alpha f_\eps\,dx
\end{equation*}
for all $\eps>0$.
On the one side, we have
\begin{equation*}
\lim_{\eps\to0^+}
\int_{\R^n}f_\eps\,
\div^\alpha(g\phi)
\,dx
=
\int_{\R^n}
f\,
\div^\alpha(g\phi)
\,dx
\end{equation*}
by H\"older's inequality in the case $p<+\infty$ and by the Dominated Convergence Theorem in the case $p=+\infty$.
On the other side, we also have
\begin{equation*}
\lim_{\eps\to0^+}
\int_{\R^n}g\phi\cdot
\nabla^\alpha f_\eps\,dx
=
\int_{\R^n}g\phi\cdot
\nabla^\alpha f\,dx
\end{equation*}
again by H\"older's inequality in the case $p<+\infty$ and by the Dominated Convergence Theorem in the case $p=+\infty$,
since $\nabla^\alpha f_\eps=\rho_\eps*\nabla^\alpha f$ in $L^p(\R^n;\R^n)$ for all $\eps>0$.
We thus conclude that 
\begin{equation*}
\int_{\R^n}f\,
\div^\alpha(g\phi)
\,dx
=
-
\int_{\R^n}g\phi\cdot
\nabla^\alpha f\,dx,
\end{equation*}
so that
\begin{equation*}
\int_{\R^n}fg\,\div^\alpha\phi\,dx 
=
-
\int_{\R^n}
\phi\cdot
\big(
g\,\nabla^\alpha f
+
f\,\nabla^\alpha g
+
\nabla^\alpha_{\rm NL}(f,g)
\big)
\,
dx
\end{equation*}
for any $\phi\in\Lip_c(\R^n;\R^n)$ and the proof is complete.
\end{proof}

\subsection{Products of two Bessel functions}

We are now ready to prove our second main result \cref{resi:lebniz_bessel_conjugate}, dealing with the product of two functions in $S^{\alpha,p}$ and $S^{\alpha,q}$ respectively, with $p,q\in(1,+\infty)$ such that $\frac1p+\frac1q\in(0,1]$.

\begin{proof}[Proof of \cref{resi:lebniz_bessel_conjugate}]
By~\cite{BCCS20}*{Theorem~A.1}, we can find $(f_k)_{k\in\N},(g_k)_{k\in\N}\subset C^\infty_c(\R^n)$ such that $f_k\to f$ in $S^{\alpha,p}(\R^n)$ and $g_k\to g$ in $S^{\alpha,q}(\R^n)$ as $k\to+\infty$.
Recalling~\cite{CS19}*{Lemma~2.6}, for each $k\in\N$ we can write
\begin{equation}
\label{eq:z04}
\nabla^\alpha(f_kg_k)
=
g_k\,\nabla^\alpha f_k
+
f_k\,\nabla^\alpha g_k
+
\nabla^\alpha_{\NL}(f_k,g_k).
\end{equation}
On the one side, we clearly have that 
\begin{equation*}
g_k\,\nabla^\alpha f_k
+
f_k\,\nabla^\alpha g_k
\to
g\,\nabla^\alpha f
+
f\,\nabla^\alpha g
\quad
\text{in}\
L^r(\R^n;\R^n)
\end{equation*}  
as $k\to+\infty$ by H\"older's inequality.
On the other side, by \cref{res:nabla_alpha_NL_div_alpha_NL_int} we can estimate
\begin{align*}
\|
\nabla^\alpha_{\NL}(f_k,g_k)
-
\nabla^\alpha_{\NL}&(f,g)
\|_{L^r(\R^n;\R^n)}
\\
&\le
\|\nabla^\alpha_{\NL}(f_k-f,g_k)
\|_{L^r(\R^n;\R^n)}
+
\|
\nabla^\alpha_{\NL}(f,g_k-g)
\|_{L^r(\R^n;\R^n)}
\\
&\le
\mu_{n,\alpha}
\big(
[f_k-f]_{B^{\frac\alpha2}_{p,2}(\R^n)}
\,
[g_k]_{B^{\frac\alpha2}_{q,2}(\R^n)}
+
[f]_{B^{\frac\alpha2}_{p,2}(\R^n)}
\,
[g_k-g]_{B^{\frac\alpha2}_{q,2}(\R^n)}
\big)
\\
&\le
c_{n,\alpha,p,q}
\big(
\|f_k-f\|_{S^{\alpha,p}(\R^n)}
\,
\|g_k\|_{S^{\alpha,q}(\R^n)}
+
\|f\|_{S^{\alpha,p}(\R^n)}
\,
\|g_k-g\|_{S^{\alpha,q}(\R^n)}
\big)
\end{align*} 
for all $k\in\N$, since one easily sees that    
$
[u]_{B^{\frac\alpha2}_{\ell,2}(\R^N)}
\le
c_{n,\alpha,\ell}
\,
\|u\|_{S^{\alpha,\ell}(\R^n)}
$
for all $u\in S^{\alpha,\ell}(\R^n)$ with $\ell\in[1,+\infty)$, thanks to~\cite{BCCS20}*{Proposition~B.2}.
Therefore also
\begin{equation*}
\nabla^\alpha_{\NL}(f_k,g_k)
\to
\nabla^\alpha_{\NL}(f,g)
\quad
\text{in}\
L^r(\R^n;\R^n)
\end{equation*}
as $k\to+\infty$ and so~\eqref{eq:z08} and~\eqref{eq:z09} readily follow.
For the case $r=1$, we observe that 
\begin{equation*}
\int_{\R^n}f_k\,\nabla^\alpha g_k\,dx
=
-
\int_{\R^n}
g_k\,\nabla^\alpha f_k\,dx
\end{equation*}
for all $k\in\N$, so that~\eqref{eq:z07} immediately follows by passing to the limit thanks to H\"older's inequality again.
Finally, the validity of~\eqref{eq:z06} comes by combining~\eqref{eq:z08} and~\eqref{eq:z07} with \cref{res:zero_total_mesure}. 
The proof is complete.  
\end{proof}

\subsection{An estimate à la Kenig--Ponce--Vega}

The rest of the present section is devoted to the proof of our third main result \cref{resi:lebniz_bessel_bounded}.
To this aim, we need the following preliminary result improving our previous integrability estimate for $\nabla^\alpha_{\NL}$ established in \cref{res:nabla_alpha_NL_div_alpha_NL_int}.

\begin{corollary}[Improved integrability of $\nabla^\alpha_{\NL}$]
\label{res:nabla_div_NL_bmo_int}
Let $\alpha\in(0,1)$ and $p\in(1,+\infty)$.
There exists a constant $c_{n,\alpha,p}>0$, depending on $n$, $\alpha$ and $p$ only, such that
\begin{equation*}
\|\nabla^\alpha_{\NL}(f,g)\|_{L^p(\R^n;\,\R^n)}
\le
c_{n,\alpha,p}
\big(
\|\nabla^\alpha f\|_{L^p(\R^n;\,\R^n)}\,[g]_{\BMO(\R^n)}
+
[f]_{\BMO(\R^n)}\,\|\nabla^\alpha g\|_{L^p(\R^n;\,\R^n)}
\big)
\end{equation*}
for all $f,g\in C^\infty_c(\R^n)$. 
\end{corollary}
    
The proof of \cref{res:nabla_div_NL_bmo_int} is a simple application of two well known integrability results for commutators involving singular integral operators.

The first result we need is the following theorem which dates back to~\cite{CRW76}*{Theorem~I}. 
For a different proof, see~\cite{LS20}*{Theorem~4.1}.  

\begin{theorem}[Coifman--Rochberg--Weiss]
\label{res:crw_estimate}
Let $p\in(1,+\infty)$.
There exists a constant $c_{n,p}>0$, depending on $n$ and $p$ only, such that
\begin{equation*}
\|R(fg)-f Rg\|_{L^p(\R^n;\,\R^n)}
\le 
c_{n,p}
\,
[f]_{\BMO(\R^n)}
\,
\|g\|_{L^p(\R^n)}
\end{equation*}
for all $f,g\in C^\infty_c(\R^n)$.
\end{theorem}

The second result we need provides two integrability estimates for the operator
\begin{equation*}
H_\alpha(f,g)
=
(-\Delta)^{\frac\alpha2}
(fg)
-
g\,(-\Delta)^{\frac\alpha2}f
-
f\,(-\Delta)^{\frac\alpha2}g
\end{equation*}
defined for all $f,g\in\Lip_c(\R^n)$.
\cref{res:ls_estimate} was originally proved in~\cite{KPV93}*{Appendix~A}.
For a different proof, as well as for an account on the related literature, see~\cite{LS20}*{Theorem~7.1}.

\begin{theorem}[Kenig--Ponce--Vega]
\label{res:ls_estimate}
Let $\alpha\in(0,1)$ and $p\in(1,+\infty)$.
There exists a constant $c_{n,\alpha,p}>0$, depending on $n$, $\alpha$ and $p$ only, such that
\begin{equation*}
\|H_\alpha(f,g)\|_{L^p(\R^n)}
\le
c_{n,\alpha,p}
\,
\|(-\Delta)^{\frac\alpha2}f\|_{L^p(\R^n)}
\,
[g]_{\BMO(\R^n)}
\end{equation*}
for all $f,g\in C^\infty_c(\R^n)$.
\end{theorem}

\begin{proof}[Proof of \cref{res:nabla_div_NL_bmo_int}]
Since $\nabla^\alpha=R(-\Delta)^{\frac\alpha2}$ on~$C^\infty_c(\R^n)$, we can write
\begin{equation*}
\begin{split}
\nabla^\alpha(fg)
-
g\,\nabla^\alpha f
-
f\,\nabla^\alpha g
&=
R
\big(
H_\alpha(f,g)
\big)
\\
&\quad+
R
\big(
g(-\Delta)^{\frac\alpha2}f
\big)
-
g
R(-\Delta)^{\frac\alpha2}f
\\
&\quad+
R
\big(
f(-\Delta)^{\frac\alpha2}g
\big)
-
f
R(-\Delta)^{\frac\alpha2}g.
\end{split}
\end{equation*}
On the one side, by \cref{res:ls_estimate} and the continuity properties of the Riesz transform, we have
\begin{equation*}
\begin{split}
\big\|
R
\big(
H_\alpha(f,g)
\big)
\big\|_{L^p(\R^n;\,\R^n)}
&\le
C\,
\|
H_\alpha(f,g)
\|_{L^p(\R^n)}
\\
&\le
C
\,
\|(-\Delta)^{\frac\alpha2}f\|_{L^p(\R^n)}
\,
[g]_{\BMO(\R^n)}
\\
&\le
C
\,
\|\nabla^\alpha f\|_{L^p(\R^n; \R^n)}
\,
[g]_{\BMO(\R^n)}.
\end{split}
\end{equation*} 
Here and in the following, $C>0$ is a constant depending on $n$, $\alpha$ and $p$ only that may change from line to line.
On the other side, by \cref{res:crw_estimate} we can estimate
\begin{align*}
\big\|
R
\big(
g(-\Delta)^{\frac\alpha2}f
\big)
-
g
R(-\Delta)^{\frac\alpha2}f
\big\|_{L^p(\R^n;\,\R^n)}
&\le
C
\,
[g]_{\BMO(\R^n)}
\,
\|(-\Delta)^{\frac\alpha2}f\|_{L^p(\R^n)}
\\
&\le
C
\,
[g]_{\BMO(\R^n)}
\,
\|\nabla^\alpha f\|_{L^p(\R^n;\,\R^n)}
\end{align*} 
and, similarly,
\begin{align*}
\big\|
R
\big(
f(-\Delta)^{\frac\alpha2}g
\big)
-
f
R(-\Delta)^{\frac\alpha2}g
\big\|_{L^p(\R^n;\,\R^n)}
\le
C
\,
[f]_{\BMO(\R^n)}
\,
\|\nabla^\alpha g\|_{L^p(\R^n;\,\R^n)}.
\end{align*} 
The conclusion immediately follows.
\end{proof}

\subsection{Products of two \texorpdfstring{$S^{\alpha,p}\cap L^\infty$}{bounded Sˆ(alpha,p)} functions}

Inspired by the duality relation between the non-local operators proved in \cref{res:div_nabla_NL_ibp}, we can now define the natural weak versions of the non-local fractional gradient and of the non-local fractional divergence.

\begin{definition}[Weak non-local $\alpha$-gradient and $\alpha$-divergence]
\label{def:weak_NL_grad}
Let $\alpha\in(0,1)$ and $p,q\in[1,+\infty]$ be such that $\frac1p+\frac1q\le1$.
Let $f\in L^p(\R^n)$, $g\in L^q(\R^n)$ and $h\in L^q(\R^n; \R^n)$. 
We say that $\nabla^\alpha_{\NL,w}(f,g) \in L^1_{\rm loc}(\R^n;\R^n)$ is a \emph{weak non-local fractional $\alpha$-gradient} of the pair $(f,g)$ if 
\begin{equation*}
\int_{\R^n}
f\,\div^\alpha_{\NL}(g,\varphi)
\,dx
=
\int_{\R^n}
\varphi\cdot \nabla^\alpha_{\NL,w}(f,g) \, dx \text{ for all } \varphi\in C^\infty_c(\R^n;\R^n),
\end{equation*}
and that $\div^\alpha_{\NL,w}(f,h) \in L^1_{\rm loc}(\R^n)$ is a \emph{weak non-local fractional $\alpha$-divergence} of the pair $(f,h)$ if 
\begin{equation*}
\int_{\R^n}
h \cdot \,\nabla^\alpha_{\NL}(f,\psi)
\,dx
=
\int_{\R^n}
\psi \, \div^\alpha_{\NL,w}(f,h) \, dx \text{ for all } \psi \in C^\infty_c(\R^n).
\end{equation*}

\end{definition}

Thanks to \cref{res:nabla_alpha_NL_div_alpha_NL_int}, it is easy to check that \cref{def:weak_NL_grad} is well posed. 
In addition, we can immediately see that, if it exists, then the weak non-local fractional gradient of the pair $(f,g)$ is unique and, thanks to \cref{res:nabla_div_NL_swap}, it must coincide with the weak non-local fractional gradient of the pair $(g,f)$.
Consequently, the resulting functional $\nabla^\alpha_{\NL,w}$ is bilinear and symmetric on its domain of definition. 
Moreover, in virtue of the integration-by-parts formula provided by \cref{res:div_nabla_NL_ibp}, we immediately see that 
\begin{equation*}
\nabla^\alpha_{\NL,w}(f,g)
=
\nabla^\alpha_{\NL}(f,g)
\quad
\text{for all}\
f\in L^p(\R^n)\
\text{and}\
g\in B^\alpha_{q,1}(\R^n).
\end{equation*}
Obviously, analogous considerations can be done also for the weak non-local fractional divergence operator~$\div^\alpha_{\NL,w}$.

Having \cref{def:weak_NL_grad} at disposal, we are now ready to prove our third main result \cref{resi:lebniz_bessel_bounded}. 

\begin{proof}[Proof of \cref{resi:lebniz_bessel_bounded}]
Since $f,g\in S^{\alpha,p}(\R^n)\cap L^\infty(\R^n)$, we can find $(f_k)_{k\in\N},(g_k)_{k\in\N}\subset C^\infty_c(\R^n)$ such that $f_k\to f$ and $g_k\to g$ in $S^{\alpha,p}(\R^n)$ as $k\to+\infty$ and, additionally, $\|f_k\|_{L^\infty(\R^n)}\le\|f\|_{L^\infty(\R^n)}$ and $\|g_k\|_{L^\infty(\R^n)}\le\|g\|_{L^\infty(\R^n)}$ for all $k\in\N$. 
This statement easily follows from the proofs of~\cite{CS19}*{Theorem~3.22} and of~\cite{KS21}*{Theorem~2.7}, so we leave the details to the interested reader.
We can thus write~\eqref{eq:z04} and argue as in the proof of \cref{resi:lebniz_bessel_conjugate}.
We define $u_k=\nabla^\alpha_{\NL}(f_k,g_k)
$ for all $k\in\N$.
By \cref{res:nabla_div_NL_bmo_int}, we know that
\begin{equation}
\label{eq:z05}
\begin{split}
\|u_k\|_{L^p(\R^n;\,\R^n)}
&\le
C
\big(
\|\nabla^\alpha f_k\|_{L^p(\R^n;\,\R^n)}\,[g_k]_{\BMO(\R^n)}
+
[f_k]_{\BMO(\R^n)}\,\|\nabla^\alpha g_k\|_{L^p(\R^n;\,\R^n)}	
\big)
\\
&\le
C
\big(
\|\nabla^\alpha f\|_{L^p(\R^n;\,\R^n)}\,\|g\|_{L^\infty(\R^n)}
+
\|f\|_{L^\infty(\R^n)}\,\|\nabla^\alpha g\|_{L^p(\R^n;\,\R^n)}	
\big)
\end{split}
\end{equation}
for all $k\in\N$ sufficiently large.
Here and in the following, $C>0$ is a constant depending on $n$, $\alpha$ and $p$ only that may change from line to line.
By the known reflexivity properties of $L^p(\R^n;\R^n)$, we can find a subsequence $(u_{k_h})_{h\in\N}$ and a vector-valued function $u\in L^p(\R^n;\R^n)$ such that
\begin{equation}
\label{eq:z01}
\lim_{h\to+\infty}
\int_{\R^n}
\phi\cdot u_{k_h}
\,
dx
=
\int_{\R^n}
\phi\cdot u
\,
dx
\end{equation}
for all $\phi\in C^\infty_c(\R^n;\R^n)$.
We now observe that
\begin{align*}
\int_{\R^n}
\phi\cdot u_k
\,dx
=
\int_{\R^n}
f_k\,\div^\alpha_{\NL}(g_k,\phi)
\,dx
\end{align*} 
for each $k\in\N$, thanks to \cref{res:div_nabla_NL_ibp}.
By \cref{res:nabla_alpha_NL_div_alpha_NL_int} and by~\eqref{eq:div_NL_swap} in \cref{res:nabla_div_NL_swap}, we can also estimate
\begin{align*}
\bigg|
\int_{\R^n}
f_k
&
\,
\div^\alpha_{\NL}(g_k,\phi)
\,dx
-
\int_{\R^n}
f\,\div^\alpha_{\NL}(g,\phi)
\,dx
\,\bigg|
\le
\|f_k-f\|_{L^p(\R^n)}
\,
\|\div^\alpha_{\NL}(g_k,\phi)\|_{L^q(\R^n)}
\\
&\quad+
\bigg|
\int_{\R^n}f\,\div^\alpha_{\NL}(g_k-g,\phi)
\,dx
\,\bigg|
\\
&\le
2
\,
\|f_k-f\|_{L^p(\R^n)}
\,
\|g_k\|_{L^\infty(\R^n)}
\,
[\phi]_{B^\alpha_{q,1}(\R^n;\,\R^n)}
+
\bigg|
\int_{\R^n}(g_k-g)\,\div^\alpha_{\NL}(f,\phi)
\,dx
\,\bigg|
\\
&\le
2
\,
[\phi]_{B^\alpha_{q,1}(\R^n;\,\R^n)}
\big(
\|f_k-f\|_{L^p(\R^n)}
\,
\|g\|_{L^\infty(\R^n)}
+
\|f\|_{L^\infty(\R^n)}
\,
\|g_k-g\|_{L^p(\R^n)}
\big)
\end{align*}
for all $k\in\N$, so that
\begin{equation}
\label{eq:z02}
\lim_{k\to+\infty}
\int_{\R^n}
\phi\cdot u_k
\,dx
=
\lim_{k\to+\infty}
\int_{\R^n}
f_k\,\div^\alpha_{\NL}(g_k,\phi)
\,dx
=
\int_{\R^n}
f\,\div^\alpha_{\NL}(g,\phi)
\,dx
\end{equation}
for all $\phi\in C^\infty_c(\R^n;\R^n)$.
From~\eqref{eq:z01}, \eqref{eq:z02} and again~\eqref{eq:div_NL_swap} in \cref{res:nabla_div_NL_swap}, we hence get that
\begin{equation}
\label{eq:z03}
\lim_{k\to+\infty}
\int_{\R^n}
\phi\cdot u_k
\,dx
=
\int_{\R^n}
\phi\cdot u
\,dx
=
\int_{\R^n}
f\,\div^\alpha_{\NL}(g,\phi)
\,dx
=
\int_{\R^n}
g\,\div^\alpha_{\NL}(f,\phi)
\,dx
\end{equation}  
for all $\phi\in C^\infty_c(\R^n;\R^n)$.
In view of~\eqref{eq:z05} and~\eqref{eq:z03}, we must have that
\begin{equation*}
u=\nabla^\alpha_{\NL,w}(f,g)
\end{equation*} 
as in \cref{def:weak_NL_grad}, with
\begin{equation*}
\|
\nabla^\alpha_{\NL,w}(f,g)
\|_{L^p(\R^n;\,\R^n)}
\le
C
\big(
\|\nabla^\alpha f\|_{L^p(\R^n;\,\R^n)}\,\|g\|_{L^\infty(\R^n)}
+
\|f\|_{L^\infty(\R^n)}\,\|\nabla^\alpha g\|_{L^p(\R^n;\,\R^n)}	
\big).
\end{equation*}
Therefore, we can integrate~\eqref{eq:z04} against an arbitrary $\phi\in C^\infty_c(\R^n;\R^n)$ and pass to the limit as $k\to+\infty$ to get that
\begin{align*}
\int_{\R^n} 
fg\,\div^\alpha\phi
\,
dx
=
\int_{\R^n} 
\phi\cdot
\big(
g\,\nabla^\alpha f
+
f\,\nabla^\alpha g
+
\nabla^\alpha_{\NL,w}(f,g)
\big)
\,
dx
\end{align*}
and the conclusion follows.
\end{proof}

\begin{remark}[The case $p=1$ in \cref{resi:lebniz_bessel_bounded}]
\label{rem:lebniz_S_alpha_1_bounded}
Concerning the validity of \cref{resi:lebniz_bessel_bounded} in the case $p=1$, we can make the following observation.
If $f,g\in S^{\alpha,1}(\R^n)\cap L^\infty(\R^n)$, then
\begin{equation}
\label{eq:leibniz_S_alpha_1_iff}
fg\in S^{\alpha,1}(\R^n)
\iff
\exists\,
\nabla^\alpha_{\NL,w}(f,g)\in L^1(\R^n;\R^n),
\end{equation}
in which case
\begin{equation}
\label{eq:leibniz_S_alpha_1_rule}
\nabla^\alpha(fg)
=
g\,\nabla^\alpha f
+
f\,\nabla^\alpha g
+
\nabla^\alpha_{\NL,w}(f,g)
\quad
\text{in}\
L^1(\R^n;\R^n).
\end{equation}
Indeed, for a given $h\in C^\infty_c(\R^n)$, by \cref{res:BV_alpha_p_leibniz_C_b} we have $gh\in S^{\alpha,1}(\R^n)\cap L^\infty(\R^n)$ with
\begin{equation*}
\nabla^\alpha(gh)
=
h\,\nabla^\alpha g
+
g\,\nabla^\alpha h
+
\nabla^\alpha_{\NL}(g,h)
\quad
\text{in}\
L^1(\R^n;\R^n).
\end{equation*}
Therefore, thanks to \cref{res:gg_S_alpha_1}, we can write
\begin{align*}
\int_{\R^n}fg\,\nabla^\alpha h\,dx
&=
\int_{\R^n}
f\,\nabla^{\alpha}(gh)\,dx
-
\int_{\R^n}
fh\,\nabla^\alpha g
\,dx
-
\int_{\R^n}
f\,\nabla^\alpha_{\NL}(g,h)\,dx
\\
&=
-\int_{\R^n}
gh\,\nabla^\alpha f\,dx
-
\int_{\R^n}
fh\,\nabla^\alpha g
\,dx
-
\int_{\R^n}
f\,\nabla^\alpha_{\NL}(g,h)\,dx
\end{align*}
for all $h\in C^\infty_c(\R^n)$ and the validity of~\eqref{eq:leibniz_S_alpha_1_iff} and~\eqref{eq:leibniz_S_alpha_1_rule} immediately follows.
However, in contrast with what happens for the case $p\in(1,+\infty)$ in \cref{resi:lebniz_bessel_bounded}, we do not know if the implication
\begin{equation*}
f,g\in S^{\alpha,1}(\R^n)\cap L^\infty(\R^n)
\implies
\exists\,\nabla^\alpha_{\NL,w}(f,g)\in L^1(\R^n;\R^n) 
\end{equation*}
is true. 
This is due to the fact that suitable extensions of \cref{res:crw_estimate} and \cref{res:ls_estimate} to the case $p=1$ are not known (and not even expected to be possible, see the discussion at the beginning of~\cite{LS20}*{Section~9}) due to the failure of the $L^1$ boundedness of the Riesz transform. 
For strictly related considerations, see the discussion in~\cite{CDH16}*{Remark~1.14}.
\end{remark}

\section{Application to elliptic fractional boundary-value problems}
\label{sec:application}

This last section is devoted to the application of the theory developed so far to the well-posedness of the boundary-value problem~\eqref{eqi:frac_BVP} for the fractional operator~$L_\alpha$ defined in~\eqref{eqi:diff_op_L_frac}.

\subsection{Product with bounded Besov functions}

We begin with the proof of the preliminary \cref{res:Leibniz_Bessel_bounded_Besov}.  

\begin{proof}[Proof of \cref{res:Leibniz_Bessel_bounded_Besov}]
We divide the proof in two steps.

\smallskip

\textit{Step~1: proof of~\eqref{eq:Leibniz_Bessel_bounded_Besov} for $f\in C^\infty_c(\R^n)$}.
Let $f\in C^\infty_c(\R^n)$ and $\phi\in\Lip_c(\R^n;\R^n)$.
Let $(\rho_\eps)_{\eps>0}$ be a family of standard mollifiers (see~\cite{CS19}*{Section~3.3} for example) and let $(g_\eps)_{\eps>0}\subset b^\alpha_{r,1}(\R^n)\cap \Lip_b(\R^n)$ be given by $g_\eps=g*\rho_\eps$ for all $\eps>0$. 
Since clearly $\Lip_b(\R^n)\subset B^\alpha_{\infty,1}(\R^n)$, by \cref{res:nabla_alpha_div_alpha_leibniz_besov}\eqref{item:div_alpha_leibniz_besov} we can write  
\begin{equation*}
\div^\alpha(g_\eps\phi)
=
g_\eps\,\div^\alpha\phi
+
\phi\cdot\nabla^\alpha g_\eps
+\div^\alpha_{\NL}(g_\eps,\phi)
\quad
\text{in}\ 
L^1(\R^n)
\end{equation*}
for all $\eps>0$.
Observing that 
\begin{equation*}
\|g_\eps\,\div^\alpha\phi\|_{L^{p_r'}(\R^n)}
\le
\|g_\eps\|_{L^\infty(\R^n)}
\,
\|\div^\alpha\phi\|_{L^{p_r'}(\R^n)},
\end{equation*}
\begin{equation*}
\|\phi\cdot\nabla^\alpha g_\eps\|_{L^{p_r'}(\R^n)}
\le
\mu_{n,\alpha}\,[g_\eps]_{B^\alpha_{r,1}(\R^n)}
\,
\|\phi\|_{L^{p'}(\R^n;\,\R^n)}
\end{equation*} 
and 
\begin{equation*}
\|\div^\alpha_{\NL}(g_\eps,\phi)\|_{L^{p_r'}(\R^n)}
\le
2\mu_{n,\alpha}
\,
[g_\eps]_{B^\alpha_{r,1}(\R^n)}
\,
\|\phi\|_{L^{p'}(\R^n;\,\R^n)},
\end{equation*}
in virtue of \cref{res:nabla_alpha_div_alpha_in_L_p} and \cref{res:nabla_alpha_NL_div_alpha_NL_int},
where $p_r=\frac{p r}{r- p}$, $p_r'=\frac{p_r}{p_r-1}$ and $p'=\frac p{p-1}$, we also get that 
\begin{equation*}
\div^\alpha(g_\eps\phi)
=
g_\eps\,\div^\alpha\phi
+
\phi\cdot\nabla^\alpha g_\eps
+\div^\alpha_{\NL}(g_\eps,\phi)
\quad
\text{in}\ 
L^{p_r'}(\R^n).
\end{equation*}
Hence we can write 
\begin{equation*}
\int_{\R^n}fg_\eps\,\div^\alpha\phi\,dx
=
\int_{\R^n}f\,\div^\alpha(g_\eps\phi)\,dx
-
\int_{\R^n}f\phi\cdot\nabla^\alpha g_\eps\,dx
-
\int_{\R^n}f\,\div^\alpha_{\NL}(g_\eps,\phi)\,dx
\end{equation*}
for all $\eps>0$. 
Since $g_\eps\phi\in\Lip_c(\R^n; \R^n)$, we can write
\begin{equation*}
\int_{\R^n}f\,\div^\alpha(g_\eps\phi)\,dx
=
-
\int_{\R^n}g_\eps\phi\cdot\nabla^\alpha f\,dx,
\end{equation*}
so that 
\begin{equation*}
\int_{\R^n}fg_\eps\,\div^\alpha\phi\,dx
=
-
\int_{\R^n}g_\eps\phi\cdot\nabla^\alpha f\,dx
-
\int_{\R^n}f\phi\cdot\nabla^\alpha g_\eps\,dx
-
\int_{\R^n}f\,\div^\alpha_{\NL}(g_\eps,\phi)\,dx
\end{equation*}
for all $\eps>0$.
Now, by the Dominated Convergence Theorem, it is easy to see that 
\begin{equation*}
\lim_{\eps\to0^+}
\int_{\R^n}fg_\eps\,\div^\alpha\phi\,dx
=
\int_{\R^n}fg\,\div^\alpha\phi\,dx
\end{equation*}
and
\begin{equation*}
\lim_{\eps\to0^+}
\int_{\R^n}g_\eps\phi\cdot\nabla^\alpha f\,dx
=
\int_{\R^n}g\phi\cdot\nabla^\alpha f\,dx
\end{equation*}
Moreover, by \cref{res:nabla_div_ibp_Besov}, we have that 
\begin{equation*}
\lim_{\eps\to0^+}
\int_{\R^n}\psi\cdot\nabla^\alpha g_\eps\,dx
=
-
\lim_{\eps\to0^+}\int_{\R^n}g_\eps\,\div^\alpha\psi\,dx
=
-
\int_{\R^n}g\,\div^\alpha\psi\,dx
\end{equation*}
for all $\psi\in\Lip_c(\R^n;\R^n)$.
However, by \cref{res:nabla_alpha_div_alpha_in_L_p} again and~\cite{Leoni17}*{Proposition~17.12},  we also have that 
\begin{equation}
\label{eq:cuscino}
\|\nabla^\alpha g_\eps\|_{L^{r}(\R^n;\,\R^n)}
\le\mu_{n, \alpha}
\,
[g_\eps]_{B^\alpha_{r,1}(\R^n)}
\le 
\mu_{n, \alpha}
\,
[g]_{B^\alpha_{r,1}(\R^n)}
\end{equation}
for all $\eps>0$. 
Therefore, possibly passing to a subsequence (which we do not relabel for simplicity), we have that 
\begin{equation*}
\lim_{\eps\to0^+}
\int_{\R^n}\psi\cdot\nabla^\alpha g_\eps\,dx
=
\int_{\R^n}\psi\cdot u\,dx
\end{equation*}
for all $\psi\in\Lip_c(\R^n;\R^n)$, for some $u\in L^{r}(\R^n;\R^n)$. 
We thus must have that 
\begin{equation*}
\int_{\R^n}\psi\cdot u\,dx
=
-
\int_{\R^n}g\,\div^\alpha\psi\,dx
\end{equation*}
for all $\psi\in\Lip_c(\R^n;\R^n)$, proving that $u=\nabla^\alpha g$ according to~\cite{CS19}*{Definition~3.19}. 
As a consequence, we can write 
\begin{equation*}
\lim_{\eps\to0^+}
\int_{\R^n}f\phi\cdot\nabla^\alpha g_\eps\,dx
=
\int_{\R^n}f\phi\cdot\nabla^\alpha g\,dx.
\end{equation*}
In a similar way, we have
\begin{equation*}
\int_{\R^n}f\,\div^\alpha_{\NL}(g_\eps,\phi)\,dx
=
\int_{\R^n}\phi\cdot\nabla^\alpha_{\NL}(f,g_\eps)\,dx
\end{equation*}
for all $\eps>0$ by \cref{res:div_nabla_NL_ibp} and
\begin{equation*}
\lim_{\eps\to0^+}
\int_{\R^n}f\,\div^\alpha_{\NL}(g_\eps,\phi)\,dx
=
\lim_{\eps\to0^+}
\int_{\R^n}g_\eps\,\div^\alpha_{\NL}(f,\phi)\,dx
=
\int_{\R^n}g\,\div^\alpha_{\NL}(f,\phi)\,dx
\end{equation*}  
by \cref{res:nabla_div_NL_swap}.
Observing that 
\begin{equation}
\label{eq:caramella}
\|\nabla^\alpha_{\NL}(f,g_\eps)\|_{L^p(\R^n;\,\R^n)}
\le
2\mu_{n,\alpha}\,\|f\|_{L^{p_r}(\R^n)}
\,
[g_\eps]_{B^\alpha_{r,1}(\R^n)}
\le
2\mu_{n,\alpha}\,\|f\|_{L^{p_r}(\R^n)}
\,
[g]_{B^\alpha_{r,1}(\R^n)}
\end{equation}
for all $\eps>0$ by \cref{res:nabla_alpha_NL_div_alpha_NL_int} and~\cite{Leoni17}*{Proposition~17.12}, we get that (up to possibly pass to a non-relabelled subsequence)
\begin{equation*}
\lim_{\eps\to0^+}
\int_{\R^n}f\,\div^\alpha_{\NL}(g_\eps,\phi)\,dx
=
\int_{\R^n}\phi\cdot\nabla^\alpha_{\NL,w}(f,g)\,dx
\end{equation*} 
according to \cref{def:weak_NL_grad}.
In conclusion, we find that 
\begin{equation*}
\int_{\R^n}fg\,\div^\alpha\phi\,dx
=
-
\int_{\R^n}g\phi\cdot\nabla^\alpha f\,dx
-
\int_{\R^n}f\phi\cdot\nabla^\alpha g\,dx
-
\int_{\R^n}\phi\cdot\nabla^\alpha_{\NL,w}(f,g)\,dx	
\end{equation*}
whenever $f\in C^\infty_c(\R^n)$, $g\in L^\infty(\R^n)\cap b^\alpha_{r,1}(\R^n)$ and $\phi\in\Lip_c(\R^n;\R^n)$.

\textit{Step~2: proof of~\eqref{eq:Leibniz_Bessel_bounded_Besov} and~\eqref{eq:Leibniz_Bessel_bounded_Besov_est} for $f\in S^{\alpha,p}(\R^n)$}.
Now let $f\in S^{\alpha,p}(\R^n)$.
We can find $(f_k)_{k\in\N}\subset C^\infty_c(\R^n)$ such that $f_k\to f$ in $S^{\alpha,p}(\R^n)$ as $k\to+\infty$.
Given $\phi\in\Lip_c(\R^n;\R^n)$, by Step~1 we can write
\begin{equation}
\label{eq:lenzuolo}
\int_{\R^n}f_kg\,\div^\alpha\phi\,dx
=
-
\int_{\R^n}g\phi\cdot\nabla^\alpha f_k\,dx
-
\int_{\R^n}f_k\phi\cdot\nabla^\alpha g\,dx
-
\int_{\R^n}\phi\cdot\nabla^\alpha_{\NL,w}(f_k,g)\,dx
\end{equation} 
Now, by the fractional Sobolev inequality (see~\cite{SS15}*{Theorem~1.8} and recall the identification $L^{\alpha,p}(\R^n)=S^{\alpha,p}(\R^n)$ proved in~\cites{BCCS20,KS21}), we know that $S^{\alpha,p}(\R^n)\subset L^{p_r}(\R^n)$ with continuous inclusion. More precisely, thanks to the interpolation of $L^p$ spaces, we have
\begin{equation} \label{eq:frac_Sob_interp}
\|f\|_{L^{p_r}(\R^n)}\le 
c_{n,\alpha,p}^{1 - \theta} 
\,
\|f\|_{L^p(\R^n)}^\theta
\,
\|\nabla^\alpha f\|_{L^p(\R^n;\,\R^n)}^{1 - \theta}
\end{equation}
for all $f\in S^{\alpha,p}(\R^n)$, where $c_{n,\alpha,p}>0$ is a constant depending on~$n$, $\alpha$ and~$p$ only, and $\theta \in [0, 1]$ satisfies $\frac{1}{p} - \frac{1}{r} = \frac{\theta}{p} + \frac{1 - \theta}{\frac{np}{n-\alpha p}}$.
Now, since $\div^\alpha_{\NL}(g,\phi)\in L^{p_r'}(\R^n)$ thanks to \cref{res:nabla_alpha_NL_div_alpha_NL_int}, we can write 
\begin{equation*}
\lim_{k\to+\infty}
\int_{\R^n}\phi\cdot\nabla^\alpha_{\NL,w}(f_k,g)\,dx
=
\lim_{k\to+\infty}
\int_{\R^n}f_k\,\div^\alpha_{\NL}(g,\phi)\,dx
=
\int_{\R^n}f\,\div^\alpha_{\NL}(g,\phi)\,dx
\end{equation*}
because $f_k\to f$ in $L^{p_\alpha}(\R^n)$ as $k\to+\infty$.
On the other side, by~\eqref{eq:caramella} in Step~1, the lower semicontinuity of the $L^p$ norm with respect to the weak convergence and \eqref{eq:frac_Sob_interp}, we immediately deduce that 
\begin{equation}
\label{eq:federa}
\begin{split}
\|\nabla^\alpha_{\NL,w}(f_k,g)\|_{L^p(\R^n;\,\R^n)}
&\le \liminf_{\eps \to 0^+} \|\nabla^\alpha_{\NL,w}(f_k,g_\eps)\|_{L^p(\R^n;\,\R^n)} \\
& \le 
2\mu_{n,\alpha}\,\|f_k\|_{L^{p_r}(\R^n)}
\,
[g]_{B^\alpha_{r,1}(\R^n)}
\\
&\le
2\mu_{n,\alpha}
\,
c_{n,\alpha,p}^{1 - \theta} 
\,
\|f_k\|_{L^p(\R^n)}^\theta
\,
\|\nabla^\alpha f_k\|_{L^p(\R^n;\,\R^n)}^{1 - \theta}
\,
[g]_{B^\alpha_{r,1}(\R^n)}
\end{split}
\end{equation}
for all $k\in\N$.
Therefore, up to possibly pass to a subsequence (which we do not relabel for simplicity), we get that
\begin{equation*}
\lim_{k\to+\infty}
\int_{\R^n}\phi\cdot\nabla^\alpha_{\NL,w}(f_k,g)\,dx
=
\int_{\R^n}\phi\cdot u\,dx
\end{equation*} 
for all $\phi\in\Lip_c(\R^n;\R^n)$, for some $u\in L^p(\R^n;\R^n)$ such that 
\begin{equation*}
\int_{\R^n}\phi\cdot u\,dx
=
\int_{\R^n}f\,\div^\alpha_{\NL}(g,\phi)\,dx.
\end{equation*}
We thus get that $u=\nabla^\alpha_{\NL,w}(f,g)$ according to \cref{def:weak_NL_grad} and so~\eqref{eq:Leibniz_Bessel_bounded_Besov} follows by passing to the limit as $k\to+\infty$ in~\eqref{eq:lenzuolo}.
The estimate~\eqref{eq:Leibniz_Bessel_bounded_Besov_est} is a plain consequence of~\eqref{eq:cuscino}, the H\"older and fractional Sobolev inequalities and~\eqref{eq:federa}, while \eqref{eq:Leibniz_Bessel_bounded_Besov_est_bis} is a trivial particular case.
The proof is complete.
\end{proof}

For the sake of completeness, below we state the analogue of \cref{res:Leibniz_Bessel_bounded_Besov} for $BV^{\alpha,p}$ functions.
In this case, the embedding $BV^{\alpha,p}(\R^n)\subset L^{\frac n{n-\alpha}}(\R^n)$ holds for $p\in\left[1,\frac n{n-\alpha}\right)$ and $n\ge2$, see~\cite{CSS21}*{Theorem~3.4}.
In the case $n=1$, we only have $BV^{\alpha,p}(\R)\subset L^q(\R)$ for all $q\in\left[p,\frac 1{1-\alpha}\right)$ whenever $p\in\left[1,\frac 1{1-\alpha}\right)$, see~\cite{CSS21}*{Theorem~3.4} again. 

\begin{theorem}[Leibniz rule for $BV^{\alpha,p}$ with bounded continuous Besov]
\label{res:Leibniz_frac_BV_bounded_Besov}
Let $\alpha\in(0,1)$, $p\in\left[1,\frac n{n-\alpha}\right)$ and $q \in \left (\frac{n}{\alpha}, + \infty \right ]$ be such that $\frac1p+\frac1q=1$. 
Also, let $r \in \left (\frac{n}{\alpha}, q \right]$, including the case $r = \frac{n}{\alpha}$ for $n \ge 2$.
There exists a constant $c_{n,\alpha,p, r}>0$ depending on~$n$, $\alpha$, $p$ and~$r$ only with the following property. 
If $f\in BV^{\alpha,p}(\R^n)$ and $g\in C_b(\R^n)\cap b^\alpha_{r, 1}(\R^n)$, then $fg\in BV^{\alpha,p}(\R^n)$ with
\begin{equation*}
D^\alpha(fg)
=
g\,D^\alpha f
+
f\,\nabla^\alpha g\,\Leb{n}
+
D^\alpha_{\NL}(f,g)
\quad
\text{in}\ \M(\R^n;\R^n)
\end{equation*}
and
\begin{equation*}
\|f\,\nabla^\alpha g\|_{L^1(\R^n;\,\R^n)}
+
|D^\alpha_{\NL}(f,g)|(\R^n)
\le
c_{n,\alpha,p,r}
\, \|f\|_{L^p(\R^n)}^\theta\,
|D^\alpha f|(\R^n)^{1 - \theta}
\, 
[g]_{B^\alpha_{r,1}(\R^n)},
\end{equation*}
where $\theta \in [0, 1]$ satisfies $1 - \frac{1}{r} = \frac{\theta}{p} + \frac{1 - \theta}{\frac{n}{n-\alpha}}$.
\end{theorem}

The proof of \cref{res:Leibniz_frac_BV_bounded_Besov} is very similar to the one of \cref{res:Leibniz_Bessel_bounded_Besov} presented above and we thus omit it.
Nonetheless, we would like to remark that, in the above statement, the assumption $g\in C_b(\R^n)\cap b^\alpha_{r,1}(\R^n)$, for $r\in\left(\frac n\alpha,q\right]$, including the case $r=\frac n\alpha$ for $n\ge2$, can be actually weakened.

In the case $r\in\left(\frac n\alpha,q\right]$, it is enough to assume that $g\in L^\infty(\R^n)\cap b^\alpha_{\frac n\alpha,1}(\R^n)$, the continuity of~$g$ being a consequence of the Morrey inequality in Besov spaces, see~\cite{Leoni17}*{Theorem~17.52}.  

In the case $r=\frac n\alpha$, for $n\ge2$, we can also just assume that $g\in L^\infty(\R^n)\cap b^\alpha_{\frac n\alpha,1}(\R^n)$, but the function~$g$ may not be continuous, so that $g\,D^\alpha f$ has to be replaced with $g^\star D^\alpha f$.   

Note that, under these assumptions, the measure $g^\star D^\alpha f$ is well defined.
Indeed, on the one side, since $p<\frac n{n-\alpha}$, we have that $|D^\alpha f|\ll\Haus{n-1}$, thanks to~\cite{CSS21}*{Theorem~1.1}.
On the other side, by the known theory on Bessel functions and \emph{fractional capacities} (see~\cite{CSS21}*{Section~5} for an account), if $g\in L^\infty(\R^n)\cap b^\alpha_{\frac n\alpha,1}(\R^n)$, then the precise representative~$g^\star$ is well defined~$\Haus{\eps}$-a.e., for any given $\eps>0$.
Indeed, it is easily verified that $g\eta\in S^{\alpha,\frac n\alpha}(\R^n)$ for any cut-off function $\eta\in C^\infty_c(\R^n)$ and, clearly, $(g\eta)^\star(x)=g^\star(x)$ provided that $\eta=1$ in an open neighborhood of the given point $x\in\R^n$.

\subsection{Well-posedness of the fractional boundary-value problem}

We are now ready to deal with the fractional operator~$L_\alpha$ introduced in~\eqref{eqi:diff_op_L_frac}.
We prove the \emph{energy estimates} for the associated the bilinear form~$\mathsf B_\alpha$ defined in~\eqref{eqi:bil_form_B}.

\begin{proof}[Proof of \cref{res:energy_est}]
By \cref{res:Leibniz_Bessel_bounded_Besov}, and, more specifically, estimate \eqref{eq:Leibniz_Bessel_bounded_Besov_est_bis}, we have that $c_1v\in S^{\alpha,2}(\R^n)$ for all $v\in S^{\alpha,2}(\R^n)$, with
\begin{equation}
\label{eq:pane}
\nabla^\alpha(c_1v)=
c_1\,\nabla^\alpha v
+
v\,\nabla^\alpha c_1
+
\nabla^\alpha_{\NL,w}(c_1,v)
\quad
\text{in}\ L^2(\R^n;\R^n)
\end{equation}
and
\begin{equation}
\label{eq:burro}
\|v\,\nabla^\alpha c_1\|_{L^2(\R^n;\,\R^n)}
+
\|\nabla^\alpha_{\NL,w}(c_1,v)\|_{L^2(\R^n;\,\R^n)}
\le
c_{n,\alpha}
\,
[c_1]_{B^\alpha_{\frac n\alpha,1}(\R^n)}
\,
\|\nabla^\alpha v\|_{L^2(\R^n;\,\R^n)},
\end{equation}
where $c_{n,\alpha}>0$ is a constant depending on~$n$ and~$\alpha$ only.
Similarly, again by \cref{res:Leibniz_Bessel_bounded_Besov} and by noticing that 
\begin{equation*}
\div^\alpha_{\NL,w}(v,b_3)
=
\sum_{j=1}^n\nabla^\alpha_{\NL,w}(v,b_3\cdot\mathrm{e}_j)
\cdot\mathrm{e}_j
\end{equation*}
in virtue of \cref{def:weak_NL_grad} (where $\mathrm{e}_1,\dots,\mathrm{e}_n$ is the canonical basis of $\R^n$), we have that $b_3 v\in S^{\alpha,2}(\R^n;\R^n)$ for all $v\in S^{\alpha,2}(\R^n)$, with
$\div^\alpha_{\NL,w}(v,b_3)\in L^2(\R^n)$ and 
\begin{equation}
\label{eq:marmellata}
\|\div^\alpha_{\NL,w}(v,b_3)\|_{L^2(\R^n)}
\le
c_{n,\alpha}
\,
[b_3]_{B^\alpha_{\frac n\alpha,1}(\R^n;\,\R^n)}
\,
\|\nabla^\alpha v\|_{L^2(\R^n;\,\R^n)}.
\end{equation}
We now prove the two estimates~\eqref{eq:energy_est_cont} and~\eqref{eq:energy_est_coercive} separately.

\smallskip

\textit{Proof of~\eqref{eq:energy_est_cont}}.
We clearly have that
\begin{align*}
\bigg|
\int_{\R^n}\nabla^\alpha v\cdot A\nabla^\alpha u\,dx
\,\bigg|
\le
\|A\|_{L^\infty(\R^n;\,\R^{n^2})}
\|\nabla^\alpha u\|_{L^2(\R^n;\,\R^n)}
\|\nabla^\alpha v\|_{L^2(\R^n;\,\R^n)}.
\end{align*}
Thanks to~\eqref{eq:pane} and~\eqref{eq:burro} above, we can estimate
\begin{align*}
\bigg|
\int_{\R^n}ub_1&\cdot\nabla^\alpha(c_1v)
\,dx
\,\bigg|
\\
&\le
\|b_1\|_{L^\infty(\R^n;\,\R^n)}
\big(
\|c_1\|_{L^\infty(\R^n)}
+c_{n,\alpha}[c_1]_{B^\alpha_{\frac n\alpha,1}(\R^n)}
\big)
\|u\|_{L^2(\R^n)}
\|\nabla^\alpha v\|_{L^2(\R^n;\,\R^n)}.
\end{align*}
Analogously, thanks to~\eqref{eq:marmellata} above, we can estimate
\begin{align*}
\bigg|
\int_{\R^n}c_3v\,\div^\alpha_{\NL,w}(u,b_3)\,dx
\,\bigg|
\le
\|c_3\|_{L^\infty(\R^n)}
[b_3]_{B^\alpha_{\frac n\alpha}(\R^n;\,\R^n)}
\|\nabla^\alpha u\|_{L^2(\R^n;\,\R^n)}
\|v\|_{L^2(\R^n)}.
\end{align*}
For the remaining terms, we easily see that 
\begin{align*}
\bigg|
\int_{\R^n}vb_2\cdot\nabla^\alpha u
\,dx
\,\bigg|
\le
\|b_2\|_{L^\infty(\R^n;\,\R^n)}
\|\nabla^\alpha u\|_{L^2(\R^n;\,\R^n)}
\|v\|_{L^2(\R^n)}
\end{align*}
and
\begin{align*}
\bigg|
\int_{\R^n}c_0uv\,dx
\bigg|
\le
\|c_0\|_{L^\infty(\R^n)}
\|u\|_{L^2(\R^n)}
\|v\|_{L^2(\R^n)},
\end{align*}
so that~\eqref{eq:energy_est_cont} readily follows by combining all the above estimates together.

\smallskip

\textit{Proof of~\eqref{eq:energy_est_coercive}}.
In virtue of~\eqref{eq:matrix_coercive}, we clearly have that  
\begin{align*}
\int_{\R^n}\nabla^\alpha u\cdot A\nabla^\alpha u\,dx
\ge
\theta
\int_{\R^n}|\nabla^\alpha u|^2\,dx.
\end{align*}
Again by~\eqref{eq:pane} and~\eqref{eq:burro}, we can also estimate
\begin{align*}
\bigg|
\int_{\R^n}ub_1&\cdot\nabla^\alpha(c_1u)
\,dx
\,\bigg|
=
\bigg|\int_{\R^n}c_1ub_1\cdot\nabla^\alpha u
\,dx
\bigg|
+
\,\bigg|
\int_{\R^n}u^2b_1\cdot\nabla^\alpha c_1
\,dx
\bigg|
\\
&\hspace*{2.85cm}
+
\bigg|
\int_{\R^n}ub_1\cdot\nabla_{\NL,w}^\alpha(c_1,u)
\,dx
\,\bigg|
\\
&\le
\|b_1\|_{L^\infty(\R^n;\,\R^n)}
\|u\|_{L^2(\R^n)}
\big(
\|c_1\|_{L^\infty(\R^n)}
\|\nabla^\alpha u\|_{L^2(\R^n;\,\R^n)}
+
\|u\,\nabla^\alpha c_1\|_{L^2(\R^n;\,\R^n)}
\\
&\quad
+
\|
\nabla^\alpha_{\NL,w}(c_1,u)\|_{L^2(\R^n;\,\R^n)}
\big)
\\
&\le
\|b_1\|_{L^\infty(\R^n;\,\R^n)}
\big(
\|c_1\|_{L^\infty(\R^n)}
+
c_{n,\alpha}[c_1]_{B^\alpha_{\frac n\alpha,1}(\R^n)}
\big)
\|u\|_{L^2(\R^n)}
\|\nabla^\alpha u\|_{L^2(\R^n;\,\R^n)}.
\end{align*}
Similarly, again by~\eqref{eq:marmellata}, we can estimate
\begin{align*}
\bigg|
\int_{\R^n}c_3u\,\div^\alpha_{\NL,w}(u,b_3)\,dx
\,\bigg|
\le
\|c_3\|_{L^\infty(\R^n)}
[b_3]_{B^\alpha_{\frac n\alpha}(\R^n;\,\R^n)}
\|u\|_{L^2(\R^n)}
\|\nabla^\alpha u\|_{L^2(\R^n;\,\R^n)}.
\end{align*}
For the remaining terms, we easily see that 
\begin{align*}
\bigg|
\int_{\R^n}ub_2\cdot\nabla^\alpha u
\,dx
\,\bigg|
\le
\|b_2\|_{L^\infty(\R^n;\,\R^n)}
\|u\|_{L^2(\R^n)}
\|\nabla^\alpha u\|_{L^2(\R^n;\,\R^n)}
\end{align*}
and
\begin{align*}
\bigg|
\int_{\R^n}c_0u^2\,dx
\bigg|
\le
\|c_0\|_{L^\infty(\R^n)}
\|u\|_{L^2(\R^n)}^2.
\end{align*}
Now we observe that
\begin{equation*}
\|u\|_{L^2(\R^n)}
\|\nabla^\alpha u\|_{L^2(\R^n;\,\R^n)}
\le
C_\eps\|u\|_{L^2(\R^n)}^2
+
\eps\|\nabla^\alpha u\|_{L^2(\R^n;\,\R^n)}^2
\end{equation*}
for all $\eps>0$, where $C_\eps>0$ is a constant depending on~$\eps$ only.
Hence we may choose $\eps>0$ sufficiently small and achieve~\eqref{eq:energy_est_coercive} by combining all the above estimates together.
\end{proof}

We are thus left to establish the well-posedness of the fractional boundary-value problem~\eqref{eqi:frac_BVP}.
The proof is a simple consequence of the Lax--Milgram Theorem.

\begin{proof}[Proof of \cref{res:boundary_problem}]
Let $\tilde f\in L^2(\R^n)$ be the extension-by-zero of $f$ outside the set~$\Omega$.
Since $S^{\alpha,2}_0(\Omega)$ is a closed subspace of $S^{\alpha,2}(\R^n)$, in particular it is a Hilbert space, so that we just need to find $u\in S^{\alpha,2}_0(\Omega)$ such that
\begin{equation*}
\mathsf B_\alpha[u,v]
+\lambda\scalar*{u,v}_{L^2(\R^n)}=
\scalar*{\tilde f, v}_{L^2(\R^n)}
\end{equation*}
for all $v\in S^{\alpha,2}_0(\Omega)$. 
In other terms, we simply have to check the assumptions of the Lax--Milgram Theorem for the bilinear form $\mathsf B_{\alpha,\lambda}\colon S^{\alpha,2}_0(\Omega) \times S^{\alpha,2}_0(\Omega)\to\R$ given by
\begin{equation*}
\mathsf B_{\alpha,\lambda}[u,v]
=
\mathsf B_\alpha[u,v]
+\lambda\scalar*{u,v}_{L^2(\R^n)}
\end{equation*}
for all $u,v\in S^{\alpha,2}_0(\Omega)$.
This follows from the energy estimates~\eqref{eq:energy_est_cont} and~\eqref{eq:energy_est_coercive} and the fractional Poincaré inequality proved in~\cite{SS15}*{Theorem~3.3} (which actually holds under the only assumption that $|\Omega| < +\infty$, see also \cite{BCM20}*{Theorem 2.2}).
The proof is complete. 
\end{proof}


\end{document}